\documentclass[a4paper]{amsart}

\usepackage[spanish,english]{babel}
\usepackage{amsmath,amssymb,amsthm}
\usepackage{latexsym}
\usepackage[pdftex,usenames,dvipsnames]{color}
\usepackage{cases}
\usepackage{enumerate}

\newtheorem{Th}{Theorem}[section]
\newtheorem{ThA}{Theorem}

\newtheorem{Cor}[Th]{Corollary}
\newtheorem{CorA}[ThA]{Corollary}
\newtheorem{Lem}[Th]{Lemma}
\newtheorem{Prop}[Th]{Proposition}
\newtheorem{Rem}[Th]{Remark}

\newcommand{\tlambda}{{\widetilde{\lambda}}}
\newcommand{\G}{\Gamma}
\newcommand{\A}{\mathbb{A}}

\newcommand{\C}{\mathbb{C}}

\newcommand{\N}{\mathbb{N}}
\newcommand{\R}{\mathbb{R}}

\newcommand{\W}{\mathbb{W}}
\DeclareMathOperator{\supp}{supp}
\DeclareMathOperator{\spann}{span}

\numberwithin{equation}{section}

\allowdisplaybreaks

\title[Fractional Bessel equation]
      {The fractional Bessel equation in H\"older spaces}

\author[J.J. Betancor]{Jorge J. Betancor}
\author[A.J. Castro]{Alejandro J. Castro}
\author[P.R. Stinga]{Pablo Ra\'ul Stinga}

\address{\newline
        Jorge J. Betancor, Alejandro J. Castro \newline
        Departamento de An\'alisis Matem\'atico,
        Universidad de la Laguna, \newline
        Campus de Anchieta, Avda. Astrof\'isico Francisco S\'anchez, s/n, \newline
        38271, La Laguna (Sta. Cruz de Tenerife), Spain}
\email{jbetanco@ull.es, ajcastro@ull.es}

\address{\newline
        Pablo Ra\'ul Stinga \newline
        Department of Mathematics, \newline
        The University of Texas at Austin, \newline
        1 University Station C1200, \newline
        78712-1202, Austin TX, United States of America}
\email{stinga@math.utexas.edu}



\thanks{Research partially supported by grants MTM2010/17974 and MTM2011-28149-C02-01
from Spanish Government. The second author is also supported by a FPU grant from Spanish Government.}

\begin{document}

  \begin{abstract}
    Motivated by the Poisson equation for the fractional Laplacian on the whole
    space with radial right hand side, we study global H\"older and Schauder estimates for a fractional Bessel equation.
    Our methods stand on the so-called semigroup language. Indeed,
    by using the solution to the Bessel heat equation we derive pointwise formulas for the fractional operators.
    Appropriate H\"older spaces, which can be seen as Campanato-type spaces,
    are characterized through Bessel harmonic extensions and fractional Carleson measures.
    From here the regularity estimates for the fractional Bessel equations follow.
    In particular, we obtain regularity estimates for radial solutions to the fractional Laplacian.
  \end{abstract}

  \maketitle

    \section{Introduction} \label{sec:intro}

    In this paper we analyze solutions to the fractional nonlocal Bessel equation
    \begin{equation}\label{Bessel eq}
    \Delta_\lambda^\sigma u=f,\quad\hbox{in}~\R_+=(0,\infty).
    \end{equation}
    {Here $\Delta_\lambda^\sigma$, $0<\sigma<1$, is the fractional power
    of the Bessel differential operator}
    $$\Delta_\lambda =-\frac{d^2}{dx^2}+\frac{\lambda
    (\lambda-1)}{x^2},\quad\lambda>0.$$
    The fractional Bessel operator $\Delta_\lambda^\sigma$ can be defined by using its spectral decomposition. Recall that
    \begin{equation}\label{spectral}
     \Delta_\lambda u=h_\lambda\big(x^2h_\lambda u\big),
    \end{equation}
    where $h_\lambda$ is the Hankel transform {and $h_\lambda^{-1}=h_\lambda$,
    see Section \ref{sec:positive}.}
    The Hankel transform on $\R_+$ plays for the Bessel operator $\Delta_\lambda$
    the same role as the Fourier transform on $\R^N$ for the Laplacian $-\Delta$.
Then, in a parallel way, we define
    \begin{equation}\label{def frac}
     \Delta_\lambda^\sigma u=h_\lambda\big(x^{2\sigma}h_\lambda u\big).
    \end{equation}
{We are interested in solutions to \eqref{Bessel eq} when $f$ belongs to the natural H\"older classes
    adapted to the problem.
    We say that a continuous function $f$ on $\R_+$ is in $C^\alpha_+$, $0<\alpha<1$,
    whenever
    the norm
    $$\|f\|_{C^\alpha_+}
        := \sup_{\substack{x,y \in \R_+\\ x \neq y}} \frac{|f(x)-f(y)|}{|x-y|^\alpha}
        + \sup_{x \in \R_+} x^{-\alpha}|f(x)|,$$
    is finite. When $\alpha=1$ we write $f\in\mathrm{Lip}_+$. It turns out that this is
    the appropriate H\"older space to look for solutions to \eqref{Bessel eq}. To establish our regularity result
we set $\widetilde{\lambda}:=\min\{\lambda,1\}$, and define
    \begin{equation}\label{Lrho}
    L_{\rho}:=\Big\{f:\R_+\to\R: \|f\|_{L_\rho}:=\int_0^\infty \frac{|f(x)|}{(1+x)^{1+2\rho}}\,dx<\infty\Big\},\quad\rho>0.
    \end{equation}}

\begin{ThA}[Schauder and H\"older estimates]\label{Th1.5}
        Let $\lambda>0$, $0<\sigma<1$ and $0<\alpha<1$.
        \begin{itemize}
            \item[$(a)$] If $\alpha+2\sigma<\widetilde{\lambda}$ and $f\in C^\alpha_+$, then
            $u=\Delta_\lambda^{-\sigma}f\in C^{\alpha+2\sigma}_+$ and
            $$\|u\|_{C^{\alpha+2\sigma}_+}\leq C\|f\|_{C^{\alpha}_+}.$$
            \item[$(b)$] If $0<\alpha-2\sigma<1$, $\lambda\geq1$ and $u\in C^\alpha_+\cap L_{\sigma}$, then
            $\Delta_\lambda^{\sigma}u\in C^{\alpha-2\sigma}_+$ and
            $$\|\Delta_\lambda^{\sigma}u\|_{C^{\alpha-2\sigma}_+}\leq C\|u\|_{C^{\alpha}_+}.$$
        \end{itemize}
    \end{ThA}

     Our motivation to study the problem \eqref{Bessel eq} is the Poisson equation
    for the fractional Laplacian
    \begin{equation}\label{fractional Laplacian}
     (-\Delta)^\sigma U=G,\quad\hbox{in}~\R^N,
    \end{equation}
    {with radial right hand side}
        $$G(X)=\phi(|X|),\quad X\in\R^N.$$
     {In the local case $\sigma=1$, a classical computation with polar coordinates shows that
     $U$ must be a radial function $U(X)=\psi(x)$, for $x=|X|$,
     verifying
     $$\Big(-\frac{d^2}{dx^2}-\frac{N-1}{x}\frac{d}{dx}\Big)\psi=\phi,\quad x\in\R_+.$$
     On the other hand, we can apply the Fourier transform to the equation. If
      $U$ is radial, then its Fourier transform is also a radial function, and
     \eqref{fractional Laplacian} with $\sigma=1$ becomes
     $$(\mathcal{H}_\lambda^{-1}\circ x^2\circ \mathcal{H}_\lambda) \psi=\phi,$$
     where $\mathcal{H}_\lambda =x^{-\lambda}\circ h_\lambda\circ x^\lambda$,
      and $2\lambda=N-1$, see \cite[p.~430]{Ste2}. Hence, the multiplier
      of the radial Laplacian with respect to the $\mathcal{H}_\lambda$--transform is $x^2$.
      Now, when $0<\sigma<1$, the same reasoning with the Fourier
      transform gives that \eqref{fractional Laplacian} is equivalent to
      $$(\mathcal{H}_\lambda^{-1}\circ x^{2\sigma}\circ \mathcal{H}_\lambda) \psi=\phi,$$
      and so, by definition, $\psi$ is a solution to}
     \begin{equation}\label{radial lambda}
     \Big(-\frac{d^2}{dx^2}-\frac{2\lambda}{x}\frac{d}{dx}\Big)^\sigma\psi=\phi,\quad x \in \R_+.
    \end{equation}
    The relation between the fractional equation \eqref{radial lambda} and our fractional
    Bessel equation \eqref{Bessel eq} is provided by the identity
    \begin{equation}\label{conjugacy}
    \Delta_\lambda^\sigma=x^\lambda\circ\Big(-\frac{d^2}{dx^2}-\frac{2\lambda}{x}\frac{d}{dx}\Big)^\sigma\circ x^{-\lambda},
    \end{equation}
    that can be easily checked with the Hankel transform $h_\lambda$. Hence,
    $$u(x):=x^\lambda\psi(x), \quad f(x):=x^\lambda\phi(x),$$
    satisfy \eqref{Bessel eq} with $2\lambda=N-1$. Conversely, the function $U(x)=\psi(|X|):=|X|^{-\lambda}u(|X|)$, where
    $u$ is a solution to \eqref{Bessel eq}, is a radial solution to the Poisson problem \eqref{fractional Laplacian}
    with $G(X)=|X|^{-\lambda}f(|X|)$.

    Therefore, the fractional nonlocal Bessel problem
    \eqref{Bessel eq} is a generalization to all $\lambda>0$ of the Poisson problem
    for the fractional Laplacian  {with radial right hand side}
    via the conjugacy with $x^{\pm\lambda}$ as in \eqref{conjugacy}.
 {Directly from Theorem~\ref{Th1.5} we obtain the following
    Schauder estimate for radial solutions to \eqref{fractional Laplacian}.}
 { \begin{CorA}[Radial solutions to the fractional Laplacian]\label{Thm:Radial}
     Let $U$ be a solution to \eqref{fractional Laplacian} with $G(X)=\phi(|X|)$ and
     $2\lambda=N-1$. If $x^\lambda\phi$ is in
     the space $C^{\alpha}_+$ defined above and $\alpha+2\sigma<\widetilde{\lambda}$,
     then $U(X)=\psi(|X|)$, for some function $\psi$ such that $x^\lambda\psi\in C^{\alpha+2\sigma}_+$, and
     $$\|x^\lambda\psi\|_{C^{\alpha+2\sigma}_+}\leq C\|x^\lambda\phi\|_{C^\alpha_+}.$$
    \end{CorA}}

    We point out that the fractional Laplacian
    on radial functions was already studied in \cite{Ferrari-Verbitsky}.  {There} a pointwise
    formula was obtained. The computation in \cite{Ferrari-Verbitsky} is based on applying polar coordinates to the formula
    $$(-\Delta)^\sigma U(X)=c_{N,\sigma} \operatorname{PV}
    \int_{\R^N}\frac{U(X)-U(Y)}{|X-Y|^{N+2\sigma}}\,dY.$$
    {In view of the underlying relation between \eqref{Bessel eq} and the radial solutions to \eqref{fractional Laplacian},
     a natural question is to get a pointwise formula for
    $\Delta_\lambda^\sigma u(x)$.}
    By our previous remarks, the pointwise formulas in \cite{Ferrari-Verbitsky} apply
    to the fractional operator in \eqref{radial lambda} for the
    case $2\lambda=N-1$. Thus, we could arrive to a pointwise formula for
     {our operator} $\Delta_\lambda^\sigma$
    passing through the conjugacy identity \eqref{conjugacy}. Nevertheless, 
     {here we obtain a pointwise formula for $\Delta_\lambda^\sigma u(x)$ for
    all values of the parameter $\lambda>0$. Our method relies on
    the so-called semigroup language approach. We
    believe it is interesting in its own right and can be applied to any operator having a heat kernel.}
    When considering the heat equation on $\R^N$ with radial data and the conjugacy with $x^{\pm\lambda}$
    as in \eqref{conjugacy} we are led to the Bessel heat equation
    \begin{equation}\label{diffusion}
    \begin{cases}
    w_{t}+\Delta_\lambda w=0, &\hbox{for}~x \in \R_+,t>0;\\
    w(x,0)=f(x),&\hbox{for}~x \in \R_+.
    \end{cases}
    \end{equation}
    It is well-known that we can write
    \begin{equation}\label{I1}
        w(x,t)\equiv W_t^\lambda f(x)
            = \int_0^\infty W_t^\lambda(x,y) f(y)\, dy,
    \end{equation}
    where the Bessel heat kernel $W_t^\lambda(x,y)$ is given explicitly in terms of a Bessel function, see \eqref{heat kernel} below.
    By using appropriate estimates for this kernel and the spectral formula
    \begin{equation}\label{3.2}
        \Delta_\lambda^\sigma u (x)
            = \frac{1}{\G(-\sigma)} \int_0^\infty \big(W_t^\lambda u(x)-u(x)\big) \frac{dt}{t^{1+\sigma}}, \quad x \in\R_+,
    \end{equation}
    we are able to prove that, for a large class of functions $u$, we have the pointwise
    expression
     $$\Delta_\lambda^\sigma u (x)
            =\operatorname{PV}\int_0^\infty \big(u(x) - u(y)\big) K_\sigma^\lambda(x,y) \, dy
                + u(x)B_\sigma^\lambda(x),\quad 0<\sigma<1.$$
    The kernel $K_\sigma^\lambda(x,y)$ and the function $B_\sigma^\lambda(x)$
    are defined with the Bessel heat kernel,
    so estimates for them can be deduced, see Theorem \ref{Thm:Puntual} below.

    {It is well known that the fractional Laplacian on $\R^N$
    can be characterized by the Caffarelli--Silvestre extension problem, see \cite{Caffarelli-Silvestre}.
    We would like to remark here that the fractional Bessel operator $\Delta_\lambda^\sigma$
    satisfies the same type of characterization.} Indeed, this is a
    direct corollary of the extension result for nonnegative selfadjoint operators of \cite{Stinga-Torrea}. Let $f$ be in the domain
    of $\Delta_\lambda^\sigma$. A solution to the extension problem
    $$\begin{cases}
    u(x,0)=f(x),&\hbox{on}~\R_+,\\
    -\Delta_\lambda u+\frac{1-2\sigma}{y}u_y+u_{yy}=0,&\hbox{in}~\R_+\times(0,\infty),
    \end{cases}$$
    is given by the Poisson formula
    $$u(x,y)=\frac{y^{2\sigma}}{4^\sigma\Gamma(\sigma)}\int_0^\infty e^{-y^2/(4t)}W_t^\lambda f(x)\,\frac{dt}{t^{1+\sigma}},$$
    and, for an explicit positive constant $c_\sigma$, we have
    $$-c_\sigma\lim_{y\to0^+}y^{1-2\sigma}u_y(x,y)=\Delta_\lambda^\sigma f(x),\quad x\in\R_+.$$
    Therefore, the fractional Bessel operator, which is a nonlocal operator, can be seen as
    the Dirichlet-to-Neumann map for a local degenerate elliptic equation in one more variable.
    Furthermore, the extension $u$ is obtained by integrating in time the Bessel heat extension of $f$.
    There are other equivalent expressions of $u$ that involve the solution of the Bessel wave equation,
    see \cite[Theorem~1.3]{Gale-Miana-Stinga}.

    Using the extension problem and the ideas of Caffarelli and Silvestre in \cite{Caffarelli-Silvestre},
    the interior Harnack's inequality for $\Delta_\lambda^\sigma$ can be proved.
    Let $f$ be a solution to $\Delta_\lambda^\sigma f=0$ in some open interval $I\subset\R_+$, with $f\geq0$ on $\R_+$.
    Then for each open bounded subinterval $I'$ compactly contained in $I$ there exists a constant $C$ depending on
    $I'$, $I$ and $\sigma$, but not on $f$, such that
    $$\sup_{I'}f\leq C\inf_{I'}f.$$
    Moreover, $f$ is H\"older continuous in $I'$. For the proof of this result see \cite[Theorem~A]{Stinga-Zhang}.
     {These extension techniques can also be applied in different
    contexts, for instance, to prove Harnack's inequality for fractional sub-Laplacians in
    Carnot groups, see \cite{Ferrari-Franchi}.}

{To close this introduction, let us explain how we are going to prove Theorem \ref{Th1.5}.
There are at least two ways of proceeding.} The first one is to use the pointwise formulas for the operators
    $\Delta_\lambda^{\pm \sigma}$ and estimate the corresponding integrals. This is the method used in \cite{SilPhD, Sil}
    for the fractional Laplacian and in \cite{STo} for the fractional harmonic oscillator.
    The second one is to use a characterization of the H\"older spaces  {$C^\alpha_+$}
    that permits to use the spectral formulas for the operators $\Delta_\lambda^{\pm \sigma}$ in a direct way.
    The characterization involves the Poisson semigroup as in \cite[Chapter~5]{Stein}. This
    idea was applied in
    \cite{MSTZ1} for fractional powers of Schr\"odinger operators $-\Delta+V$, with $V\geq0$ in the
    reverse H\"older class, and in \cite{Roncal-Stinga} for the fractional Laplacian on the torus.
    In this paper we will apply the second method.

    {The starting point is to study
    the Bessel harmonic extension problem in $C^\alpha_+$, namely, solutions to
    \begin{equation}\label{Bessel harmonic extension}
    \begin{cases}
    v_{tt}-\Delta_\lambda v=0, &\hbox{for}~x \in \R_+,t>0,\\
    v(x,0)=f(x),&\hbox{for}~x \in \R_+,
    \end{cases}
    \end{equation}
    when $f\in C^\alpha_+$.
    It is well known that $v=P_t^\lambda f$ is given by
    the Poisson semigroup associated to $\Delta_\lambda$. In fact, when $f\in L^p(\R_+)$,
    the function $v$
    can be written through the classical
     \textit{Bochner subordination formula}, see \cite[pp.~46--47]{Stein Topics}.
    In Section \ref{sec:Poisson} we analyze \eqref{Bessel harmonic extension}
    for $f$ in $C^\alpha_+$ and
    we also give necessary estimates for the Poisson kernel.}

{The second step consists in describing the spaces $C^\alpha_+$
in terms of $P_t^\lambda$. In Theorem \ref{Th1.1} we show, among other
characterizations, that a function $f$ is in $C^\alpha_+$, $0<\alpha<1$, if and only if
\begin{equation}\label{cara}
\|t^\beta\partial_t^\beta P_t^\lambda f\|_{L^\infty(\R_+)}\leq Ct^\alpha,\quad t>0,
\end{equation}
for some constant $C$ and all $\beta$ such that $\min\{\beta,\lambda\}>\alpha$.
Here $\partial_t^\beta$, $\beta>0$, is the fractional derivative
as defined by C. Segovia and R. L. Wheeden
in \cite{SW}, see \eqref{fractional derivative} below. Notice that $\partial_t^\beta$
coincides with the usual derivative when $\beta$ is a positive integer.}

{Finally, Theorem A will be proved by using \eqref{cara} and the identities
    \begin{equation}\label{spectral positive}
     \Delta_\lambda^{\sigma}u(x)
        = \frac{1}{\G(-2\sigma)} \int_0^\infty  \big(P_t^\lambda u(x)-u(x)\big) \frac{dt}{t^{1+2\sigma}},\quad
        0<\sigma<1/2,
    \end{equation}
    and
    \begin{equation}\label{spectral negative}
     \Delta_\lambda^{-\sigma}f(x)
        = \frac{1}{\G(2\sigma)} \int_0^\infty  P_t^\lambda f(x)\, \frac{dt}{t^{1-2\sigma}},\quad
        \sigma>0,
    \end{equation}
    see Section \ref{sec:Poisson and fractional}.}

    All the results and methods used in this paper open the way to more problems in PDEs and Harmonic
    and Functional Analysis.
    A particularly interesting problem is to show the boundary Harnack's inequality for $\Delta_\lambda^\sigma$
    by using the ideas in \cite{Caffarelli-Silvestre} (see \cite{Roncal-Stinga} for the case of the fractional
    Laplacian on the torus). Since the heat kernel of the Bessel operator in \eqref{radial lambda} is also
    available, the semigroup approach we use here could also be applied to obtain a generalization to
    all $\lambda>0$ of the pointwise formula of \cite{Ferrari-Verbitsky}.
    As for questions arising from Theorem \ref{Th1.1},
    notice that in Section \ref{sec:Hardy} we need to introduce an atomic Hardy space $H^p(\R_+)$ associated to
    the Bessel operator whose dual is $C^\alpha_+$. An open question is the characterization of
    this new $H^p(\R_+)$ space by using Bessel--heat or Bessel--Poisson maximal operators, or Bessel--Riesz transforms.

    The paper is organized as follows. In Section \ref{sec:positive} we give the pointwise integro-differential
    formula for $\Delta_\lambda^{\sigma}u(x)$ and,  {as a corollary}, we study the pointwise limits
    of $\Delta_\lambda^\sigma u(x)$ as $\sigma\to0^+$ and $\sigma\to1^-$.
    Section \ref{sec:Poisson} is devoted to the analysis of the Bessel harmonic extension of functions
    in $C^\alpha_+$.
    The atomic Hardy spaces are considered in Section~\ref{sec:Hardy}, where it is shown that
    their duals are the $C_+^\alpha$ spaces in its Campanato-type form.
     {The proof of the characterizations of $C^\alpha_+$
    is} given in Section \ref{sec:core},
     {while that of Theorem \ref{Th1.5} is presented in Section \ref{sec:Poisson and fractional}.}
    The last section of this paper is devoted to collect the proofs of some technical results
    left open in the previous sections.

    Throughout this paper by $C$ and $c$ we always denote positive constants that may change in each occurrence.
    Without mentioning it, we will repeatedly apply the inequality $x^\eta e^{-x}\leq C_\eta e^{-x/2}$, $\eta\geq0$,
    $x \in \R_+$.

    \section{Pointwise formula for the fractional Bessel operator} \label{sec:positive}

    As we pointed out in the introduction,
the Hankel transform given in \eqref{Hankel} below plays for the
    Bessel operator $\Delta_\lambda$ on $\R_+$
 the same role as the Fourier transform for the
    Laplacian $-\Delta$ on $\R^N$.  For a function $u\in L^1(\R_+)$, the Hankel
transform $h_\lambda$ is defined by
\begin{equation}\label{Hankel}
    h_\lambda u(x)= \int_0^\infty \sqrt{xy} J_{\lambda-1/2}(xy) u(y)\,dy, \quad x \in \R_+.
    \end{equation}
    Here $J_\nu$ is the usual Bessel function. It is known that $h_\lambda$
    can be extended as an isometry on $L^2(\R_+)$
 with $h^{-1}_\lambda=h_\lambda$, see \cite{Ze} for this and more properties.
We need the parallel of the Schwartz space of test functions $\mathcal{S}$.
    The natural space of test functions for $h_\lambda$ is
    the set $\mathcal{S}_\lambda$, which consists of all those smooth functions $\phi$ on $\R_+$ such that,
    for every $m,k \in \N_0$,
    $$\gamma_{m,k}(\phi)
        =\sup_{x \in \R_+} (1+x)^m \Bigg| \left( \frac{1}{x} \frac{d}{dx}\right)^k \left( x^{-\lambda}\phi(x)\right)\Bigg| < \infty.$$
    Observe that we ask for the function $x^{-\lambda}\phi(x)$ to have all its ``derivatives''
    $(\frac{1}{x}\frac{d}{dx})^k(x^{-\lambda}\phi)$ to be rapidly decreasing at
    infinity. A typical example of function in the space $\mathcal{S}_\lambda$ is $\phi(x)=x^\lambda e^{-x^2}$,
    $x \in\R_+$.

      \begin{Rem}\label{rem:test acotadas}
    {\rm It follows through an induction argument that all the usual derivatives of $\phi\in \mathcal{S}_\lambda$
    have rapid decay at infinity, that is, $x^m\phi^{(k)}(x)\to0$ as $x\to\infty$ for every $m,k\in\mathbb{N}$.
    Functions in $\mathcal{S}_\lambda$ are always bounded since
    $$|\phi(x)|=x^\lambda|x^{-\lambda}\phi(x)|\leq
     \gamma_{m,0}(\phi),\quad \hbox{as soon as}~m\geq\lambda.$$
    Nevertheless, their derivatives $\phi^{(k)}$ may not be
    bounded at $x=0$ for every $\lambda>0$. By computing
    $(\frac{1}{x}\frac{d}{dx})(x^{-\lambda}\phi(x))$ we readily see that
    $$|\phi'(x)|\leq \lambda x^{\lambda-1}|x^{-\lambda}\phi(x)|+x^{\lambda+1}\left|\left(\frac{1}{x}
    \frac{d}{dx}\right)(x^{-\lambda}\phi(x))\right|\leq C\big(\gamma_{m,0}(\phi)+\gamma_{m,1}(\phi)\big),$$
    as soon as $\lambda\geq1$ and we choose $m\geq\lambda+1$.
    Developing $(\frac{1}{x}\frac{d}{dx})^2(x^{-\lambda}\phi(x))$ we also get
    \begin{align*}
    |\phi''(x)| &\leq \lambda(\lambda+2)x^{\lambda-2}|x^{-\lambda}\phi(x)|
    +(2\lambda+1)x^\lambda\left|\left(\frac{1}{x}
    \frac{d}{dx}\right)(x^{-\lambda}\phi(x))\right|
+x^{\lambda+2} \Bigg|\left(\frac{1}{x}
    \frac{d}{dx}\right)^2(x^{-\lambda}\phi(x))\Bigg|\\
    &\leq C\big(\gamma_{m,0}(\phi)+\gamma_{m,1}(\phi)+\gamma_{m,2}(\phi)\big),
    \end{align*}
    whenever $\lambda\geq2$ and $m\geq\lambda+2$.}
    \end{Rem}

    If $\mathcal{S}_\lambda$ is endowed with the topology induced by
    the set of seminorms $\{\gamma_{m,k}\}_{m,k \in \mathbb{N}_0}$ above, then it becomes a Fr\'echet space and the Hankel transform $h_\lambda$ of \eqref{Hankel}
    is an automorphism of $\mathcal{S}_\lambda$, see \cite[Lemma 8]{Ze}. Moreover, $\mathcal{S}_\lambda$
    is dense in $L^p(\R_+)$, $1\leq p<\infty$, because it contains $C_c^\infty(\R_+)$, the space of smooth functions with compact support on $\R_+$.
    The dual space of $\mathcal{S}_\lambda$ is denoted by $\mathcal{S}_\lambda'$ and the Hankel transformation is defined
    on it by transposition.

    For every test function $\phi\in\mathcal{S}_\lambda$, $\Delta_\lambda\phi=h_\lambda\left( x^2 h_\lambda \phi\right)$,
    see \cite[(12)]{Ze}.
    By the isometry property of $h_\lambda$ on $L^2(\R_+)$,
    we can extend the definition of $\Delta_\lambda$
    as a positive and selfadjoint operator via \eqref{spectral}, that we still denote by $\Delta_\lambda$,
    in the domain
    $D(\Delta_\lambda)=\{u\in L^2(\R_+) : x^2 h_\lambda u\in L^2(\R_+) \} \subset L^2(\mathbb{R}_+)$.
    Moreover $C_c^\infty(\R_+) \subset D(\Delta_\lambda)$.

    The Bessel heat-diffusion equation \eqref{diffusion} can then be solved with the Hankel transform:
    \begin{equation}\label{defcalor}
        w(x,t)\equiv W_t^\lambda f(x)
            = h_\lambda\left( e^{-ty^2} h_\lambda f(y)\right)(x), \quad f \in L^2(\R_+).
    \end{equation}
    It is known that the semigroup $W_t^\lambda$ is an integral operator as in \eqref{I1}.
    The Bessel heat kernel is
    \begin{equation}\label{heat kernel}
     W_t^\lambda(x,y)
        = \frac{\sqrt{xy}}{2t} I_{\lambda-1/2}\left( \frac{xy}{2t} \right) e^{-(x^2+y^2)/(4t)},
        \quad  x,y\in\R_+,~t>0,
    \end{equation}
    where $I_\nu$ is the modified Bessel function of the first kind and order $\nu$, see \cite[\S13.31(1)]{Wat}.
    Using this kernel it can be seen that \eqref{I1} is the classical solution of the equation in \eqref{diffusion}, whenever $f$ in $L^p(\R_+)$, $1\leq p\leq\infty$. Moreover, $\big\|\sup_{t>0}|W_t^\lambda f|\big\|_{L^p(\R_+)}\leq C\|f\|_{L^p(\R_+)}$, for $1<p<\infty$,
    see \cite[Theorem 2.1 and Remark 3.2]{BHNV}.  We set, for $x\in\R_+$, $t>0$,
$$ W_t^\lambda1(x)=\int_0^\infty W_t^\lambda(x,y)\,dy.$$

    Next, we define
    $$\Delta_\lambda^\sigma \phi(x)= h_\lambda \left( y^{2\sigma}(h_\lambda\phi)(y)\right)(x),\quad \phi\in \mathcal{S}_\lambda,~0<\sigma<1.$$
    Note that if $\phi \in \mathcal{S}_\lambda$ then $\Delta_\lambda^\sigma \phi \in C^\infty(\R_+)$ but not necessarily
    $\Delta_\lambda^\sigma \phi \in \mathcal{S}_\lambda$ because we cannot guarantee that $\Delta_\lambda^\sigma \phi$ decays
    rapidly at infinity. Using the identity
    $$y^{2\sigma}
        = \frac{1}{\G(-\sigma)} \int_0^\infty \big( e^{-ty^2}-1 \big) \,\frac{dt}{t^{1+\sigma}}, \quad y \in \R_+,$$
    in the definition above, Fubini's theorem (note that the function $\sqrt{z}J_\nu(z)$ is
    bounded on $\R_+$ for every $\nu>-1/2$)
    and \eqref{defcalor} we see that \eqref{3.2} is valid
    by replacing $u$ by
    every $\phi \in \mathcal{S}_\lambda$ and $x \in \R_+$.

    The idea to obtain the pointwise formula for $\Delta_\lambda^\sigma \phi (x)$ is to
    write down the heat kernel in \eqref{3.2} and to apply Fubini's theorem. To pass to more general
    H\"older continuous functions $u$ we will use an approximation argument.

    \begin{Th}\label{Thm:Puntual}
     Let $\lambda >0$, $0<\sigma<1$ and $u\in L_\sigma$. Define the kernel
     $$K_\sigma^\lambda(x,y)
            = \frac{1}{-\G(-\sigma)} \int_0^\infty W_t^\lambda(x,y)\, \frac{dt}{t^{1+\sigma}}, \quad x,y\in\R_+, $$
        and the function
      $$B_\sigma^\lambda(x)
            = \frac{1}{\G(-\sigma)} \int_0^\infty \big(W_t^\lambda1(x) - 1 \big)\,\frac{dt}{t^{1+\sigma}},\quad x\in\R_+.$$
    Then both integrals above are absolutely convergent,
    \begin{equation}\label{estimacion nucleo}
     0\leq K_\sigma^\lambda(x,y) \leq \frac{C}{|x-y|^{1+2\sigma}},\quad x\neq y,
    \end{equation}
    and
     \begin{equation}\label{20.1}
        |B_\sigma^\lambda(x)|\leq\frac{C}{x^{2\sigma}}.
        \end{equation}
     Moreover, the following assertions are true.
     \begin{enumerate}[$(1)$]
      \item Let $0<\sigma<1/2$ and $\lambda\geq1$. If $u\in C_+^\alpha$, with $0<2\sigma<\alpha < 1$, or $u\in \mathrm{Lip}_+$ when $\alpha=1$, then
      $\Delta_\lambda^\sigma u$ is a continuous function on $\R_+$ given by
       \begin{equation}\label{10.1}
        \Delta_\lambda^\sigma u(x)
            =\int_0^\infty \big(u(x) - u(y)\big)K_\sigma^\lambda(x,y)\,dy
                + u(x)B_\sigma^\lambda(x), \quad x\in\R_+,
        \end{equation}
        where the integral is absolutely convergent.
      \item Let $1/2\leq\sigma<1$ and $\lambda\geq2$. If $u$ is continuous on $[0,\infty)$ and it has a derivative $u'\in C_+^\alpha$,  with $0\leq2\sigma-1<\alpha < 1$, or $u'\in \mathrm{Lip}_+$ when $\alpha=1$,
      then $\Delta_\lambda^\sigma u$ is a continuous function on $\R_+$ given by
        \begin{equation}\label{19.1}
            \Delta_\lambda^\sigma u(x)
                =\lim_{\varepsilon \to 0^+}\int_{0,\,|x-y|>\varepsilon}^\infty \big(u(x) - u(y)\big)K_\sigma^\lambda(x,y)\,dy
                + u(x)B_\sigma^\lambda(x), \quad x\in\R_+.
        \end{equation}
     \end{enumerate}
    \end{Th}

    To prove this result we need some properties of the Bessel function $I_\nu$,
    $\nu>-1$, that appears in the Bessel heat kernel. See \cite[pp. 108, 123, 110]{Leb} for the following:
    \begin{equation}\label{I origen}
        z^{-\nu}I_\nu(z) \sim \frac{1}{2^\nu \G(\nu+1)}, \quad \text{ as } z \to 0^+,
    \end{equation}
    \begin{equation}\label{I infinito}
        \sqrt{z}I_\nu(z) = \frac{e^z}{\sqrt{2\pi}} \left( 1 + \Psi_\nu(z) \right), \quad z \in \R_+,
    \end{equation}
    being $\Psi_\nu(z)= -(4\nu^2-1)/(8z) + \mathcal{O}(1/z^2) = \mathcal{O}(1/z)$, as $z \to \infty$, and
    \begin{equation}\label{I derivada}
        \frac{d}{dz}\left(  z^{-\nu}I_\nu(z) \right)
            =  z^{-\nu}I_{\nu+1}(z), \quad z \in \R_+.
    \end{equation}

  By taking into account \eqref{I origen} and \eqref{I infinito} we obtain that for $x,y\in\R_+$, $t>0$,
 \begin{equation}\label{W origen}
W_t^\lambda(x,y) \sim \frac{(xy)^\lambda}{t^{\lambda+1/2}} e^{-(x^2+y^2)/(4t)}, \quad\hbox{when}~\dfrac{xy}{2t}<1,
\end{equation}
\begin{equation} \label{W infinito}
W_t^\lambda(x,y) = \frac{e^{-|x-y|^2/(4t)}}{\sqrt{4 \pi t}}
\left(1+\mathcal{O}\Big(\frac{t}{xy}\Big)\right),\quad\hbox{when}~ \dfrac{xy}{2t}\geq 1.
\end{equation}
From here we readily deduce that
    \begin{equation}\label{control clasico}
        W_t^\lambda(x,y) \leq C \W_t(x-y), \quad x,y\in\R_+,~t>0,
    \end{equation}
where $\W_t(x)=(4\pi t)^{-1/2}e^{-|x|^2/(4t)}$, $x\in\R$, $t>0$,  is the Gauss--Weierstrass
kernel for the classical heat semigroup on the real line $\mathbb{W}_tf(x)$, $f\in L^2(\R)$. From \eqref{I infinito} we deduce that,
        \begin{equation}\label{25.1}
            W_t^\lambda(x,y) - \W_t(x-y) =  \W_t(x-y) \left[ - \lambda (\lambda-1) \frac{t}{xy} +
                \mathcal{O}\left( \Big(\frac{t}{xy}\Big)^2 \right)\right], \quad x,y \in \R_+, \ t>0.
        \end{equation}
Notice that for $x\in\R$ and $t>0$, $\displaystyle\W_t1(x)=\int_\R\W_t(x-y)\,dy =1$.

    Let us begin with the following

    \begin{Lem}\label{Lem:funcion}
     Estimates \eqref{estimacion nucleo} and \eqref{20.1} hold. Moreover, $B_\sigma^\lambda(x)$ is a continuous function on $\R_+$.
    \end{Lem}

    \begin{proof}
    The first one is an immediate consequence of \eqref{control clasico} and the change of variables $r=\frac{|x-y|^2}{4t}$:
    \begin{equation}\label{cuenta clasica}
    0\leq K_\sigma^\lambda(x,y)
            \leq C \int_0^\infty \frac{e^{-|x-y|^2/(4t)}}{t^{1/2}}\frac{dt}{t^{1+\sigma}}
            =\frac{C}{|x-y|^{1+2\sigma}}.
     \end{equation}

    For \eqref{20.1}, observe that we can write
        $$B_\sigma^\lambda(x)
               =\frac{1}{\G(-\sigma)} \int_0^\infty \int_0^\infty \big(W_t^\lambda(x,y)- \W_t(x-y) \big) \frac{dy\,dt}{t^{1+\sigma}}
                        - \frac{1}{\G(-\sigma)} \int_0^\infty \int_0^\infty \W_t(x+y) \frac{dy\,dt}{t^{1+\sigma}}. $$
        So we have $B_{\sigma}^\lambda(x)=:A_1(x)-A_2(x)$. Let us estimate the second term:
       $$|A_2(x)| = C \int_0^\infty \int_0^\infty e^{-(x+y)^2/(4t)}\frac{dt\,dy}{t^{\sigma+3/2}}
                = C \int_0^\infty \frac{dy}{(x+y)^{2\sigma+1}} \int_0^\infty e^{-s}s^\sigma\frac{ds}{s^{1/2}}
                =\frac{C}{x^{2\sigma}}.$$
       As for the first term, from \eqref{control clasico} and \eqref{25.1} it follows that
        \begin{align*}
            |A_1(x)|
                & \leq C \left( \int_0^\infty \int_0^{t/x} \W_t(x-y) \frac{dy\,dt}{t^{1+\sigma}}
                +\int_0^\infty \int_{t/x}^\infty  \W_t(x-y) \frac{t}{xy}\frac{dy\,dt}{t^{1+\sigma}}\right)\\
                & =: C \big(A_{1,1}(x) + A_{1,2}(x) \big).
        \end{align*}
        Now
        \begin{align*}
            A_{1,1}(x)
                &= C\int_0^\infty \Bigg( \int_{(0,t/x)\cap (0,x/2)} + \int_{(0,t/x)\cap (x/2,2x)} + \int_{(0,t/x)\cap (2x,\infty)}\Bigg)
                            e^{-|x-y|^2/(4t)} \frac{dy\,dt}{t^{\sigma + 3/2}} \\
                &\leq C\int_0^\infty \int_{(0,t/x)\cap (0,x/2)}  e^{-cx^2/t} \frac{dy\,dt}{t^{\sigma + 3/2}}
                 + C\int_{x/2}^{2x}\int_{xy}^\infty \frac{dt\,dy}{t^{\sigma + 3/2}} \\
                &\quad + C \int_0^\infty \int_{(0,t/x)\cap (2x,\infty)}  e^{-cy^2/t} \frac{dy\,dt}{t^{\sigma + 3/2}} \\
                &\leq Cx\int_0^\infty \frac{e^{-cx^2/t}}{t^{\sigma + 3/2}} \,dt
                              +   C\int_{x/2}^{2x} \frac{dy}{(xy)^{\sigma + 1/2}}
                              +C \int_{2x^2}^\infty  \int_0^\infty \frac{e^{-cz^2}}{t^{\sigma + 1}} \,dz\, dt
                =\frac{C}{x^{2\sigma}}.
        \end{align*}
        By proceeding similarly,
        \begin{align*}
            A_{1,2}(x)
               & \leq  C \int_0^\infty \int_{(t/x,\infty)\cap (0,x/2)}\frac{e^{-cx^2/t}}{y^2} \frac{dy\,dt}{t^{\sigma + 1/2}}
                             + C\int_0^\infty \int_{(t/x,\infty)\cap (2x,\infty)}  \frac{e^{-cy^2/t}}{x^2} \frac{dy\,dt}{t^{\sigma + 1/2}} \\
                     &  \quad+ C \int_{x^2}^\infty \int_{(t/x,\infty)\cap (x/2,2x)}\frac{dy\,dt}{y^2t^{\sigma + 1/2}}
                             + \frac{C}{x^2} \int_0^{x^2} \int_{(t/x,\infty)\cap (x/2,2x)}\W_t(x-y)\,dy\,\frac{dt}{t^\sigma} \\
              &  \leq C x \int_0^\infty \frac{e^{-cx^2/t}}{t^{\sigma + 3/2}} \,dt
                              + C\int_0^\infty \frac{e^{-cx^2/t}}{t^{\sigma + 1}}\, dt
                              + Cx \int_{x^2}^\infty \frac{dt}{t^{\sigma + 3/2}}
                              + \frac{C}{x^2} \int_0^{x^2} \frac{dt}{t^{\sigma }}=\frac{C}{x^{2\sigma}}.
        \end{align*}
        Putting together the estimates above we get \eqref{20.1}.

        Now we prove the continuity of $B_\sigma^\lambda$ in $\R_+$. Let $x_0 \in \R_+$. We have that
        \begin{align*}
             \lim_{x \to x_0} \frac{1}{\G(-\sigma)} \int_0^\infty \big(W_t^\lambda1(x) - 1 \big)\,\frac{dt}{t^{1+\sigma}}
                & = \frac{1}{\G(-\sigma)} \int_0^\infty \Big(\lim_{x \to x_0}  W_t^\lambda1(x) - 1 \Big)\,\frac{dt}{t^{1+\sigma}} \\
                & = \frac{1}{\G(-\sigma)} \int_0^\infty \big(  W_t^\lambda1(x_0) - 1 \big)\,\frac{dt}{t^{1+\sigma}}
                  = B_\sigma^\lambda(x_0).
        \end{align*}
        Indeed, the first equality follows from the dominated convergence theorem
        because the integral is absolutely convergent, see Theorem \ref{Thm:Puntual} and the
        proof of Lemma \ref{Lem:funcion}.
        The second one is true provided that
        $W_t^\lambda1$ is a continuous function in $x_0$. To see this last fact, we proceed by taking limits as above. From the definition
        \eqref{heat kernel}, $W_t^\lambda(\cdot, y)$ is continuous for every $t,y \in \R_+$. Moreover, for every $a>0$,
        and application of \eqref{control clasico} leads to
        \begin{align*}
            W_t^\lambda(x,y)
                & \leq \frac{C}{\sqrt{t}} \Big( \chi_{(0,2a)}(y) + \chi_{(2a, \infty)}(y) e^{-cy^2/t}  \Big)
                =: g_t(y), \quad t,y \in \R_+, \ x \in (0,a),
        \end{align*}
        and $g_t \in L^1(0,\infty)$, for every $t \in \R_+$.

    \end{proof}

    Next we get the pointwise formulas for test functions.

    \begin{Prop}\label{Prop_representacion}
     Let $\phi\in\mathcal{S}_\lambda$. Then, for every $x\in\R_+$, \eqref{3.2} holds when $\phi$ replaces $u$.
     Furthermore,
     the pointwise formulas \eqref{10.1} and \eqref{19.1} of Theorem \ref{Thm:Puntual} are valid for $\phi$ in place of $u$
     and without any restriction on $\lambda>0$.
    \end{Prop}

    \begin{proof}
    As it was explained prior to Theorem \ref{Thm:Puntual}, \eqref{3.2} is true
    when $\phi$ replaces $u$ and every $x\in\R_+$.

    For the second part of the proposition, fix any $\lambda>0$ and $x \in \R_+$.
    Assume that $0<\sigma<1/2$. According to \eqref{3.2} and \eqref{I1} we can write
        \begin{align*}
            \Delta_\lambda^\sigma \phi (x)
                &= \frac{1}{\G(-\sigma)} \int_0^\infty \big(W_t^\lambda \phi(x)- \phi(x)W_t^\lambda 1(x)\big)\frac{dt}{t^{1+\sigma}}
                   + \frac{\phi(x)}{\G(-\sigma)} \int_0^\infty \big( W_t^\lambda1(x)- 1 \big) \frac{dt}{t^{1+\sigma}} \\
                &= \frac{1}{\G(-\sigma)} \int_0^\infty\int_0^\infty W_t^\lambda(x,y)\big( \phi(y)- \phi(x)\big)\frac{dy\,dt}{t^{1+\sigma}}
                   + \phi(x)B_\sigma^\lambda(x).
        \end{align*}
        Let us see that the first integral is absolutely convergent. By using \eqref{control clasico},
        the second part in \eqref{cuenta clasica}, the mean value theorem and the facts that
        $\phi$ is bounded (see Remark \ref{rem:test acotadas}) and $\phi'$ --being smooth-- is bounded in a neighborhood of $x$,
       \begin{align*}
             \int_0^\infty&\int_0^\infty W_t^\lambda(x,y)\big| \phi(y)- \phi(x)\big|\frac{dy\,dt}{t^{1+\sigma}}\\
              &  \leq C\Bigg(\int_0^{x/2}~+\int_{x/2}^{2x}~+\int_{2x}^\infty~\Bigg)\int_0^\infty
                    \mathbb{W}_t(x-y)\,\frac{dt}{t^{1+\sigma}}|\phi(y)-\phi(x)|\,dy \\
            & \leq C\Bigg(\int_0^{x/2}~+\int_{2x}^\infty~\Bigg)\frac{dy}{|x-y|^{1+2\sigma}}
                +C\int_{x/2}^{2x}\frac{|x-y|}{|x-y|^{1+2\sigma}}\,dy=C_{\sigma,x,\phi}.
       \end{align*}
       The last integral is finite because $2\sigma<1$. Hence, by Fubini's theorem the result follows with $K_\sigma^\lambda(x,y)$
       as in the statement.

        Suppose now that $1/2 \leq \sigma < 1$. We extend the domain of definition of $\phi\in \mathcal{S}_\lambda$
        to the whole line by putting $\phi(z)\equiv0$ for $z \leq 0$, and we call this new function $\phi_0(z)$.
        We also set $W_t^\lambda(x,y)\equiv0$ for $y\leq0$.
        Let us add and subtract $\mathbb{W}_t\phi_0(x)$ in \eqref{3.2} (with $\phi$ replacing $u$) to get,
        \begin{align*}
            \Delta_\lambda^\sigma \phi (x)
                = & \frac{1}{\G(-\sigma)} \int_0^\infty \big(W_t^\lambda\phi(x)- \W_t\phi_0(x)\big) \frac{dt}{t^{1+\sigma}}
                   + \frac{1}{\G(-\sigma)} \int_0^\infty \big(\W_t\phi_0(x) - \phi(x) \big)\frac{dt}{t^{1+\sigma}}.
        \end{align*}
        Since $\phi$ is a bounded function on $\R_+$, by writing down the semigroups
        with their respective kernels and proceeding as with $A_1(x)$ in the proof of Lemma \ref{Lem:funcion} above we get
        that the first integral is absolutely convergent and, moreover, it is bounded by $C/x^{2\sigma}$.
        Also, by taking into account Taylor's formula for $\phi$ around $x \in \R_+$
        (see for example  \cite[p.~9]{SilPhD} and \cite[Lemma~5.1]{Stinga-Torrea}),
        $$\int_0^\infty \big(\W_t\phi_0(x) - \phi(x)\big)\frac{dt}{t^{1+\sigma}}
            = \lim_{\varepsilon \to 0^+} \int_{|x-y|>\varepsilon} \int_0^\infty \W_t(x-y)
            \big( \phi_0(y)-\phi(x) \big) \frac{dt\,dy}{t^{1+\sigma}}.$$
        We remark that in order to apply Taylor's formula above we only need $\phi''$ to
        be bounded in a small neighborhood of $x$, which in this case is true because $\phi$ is smooth.
        Hence, with this and Fubini's theorem we conclude that
        \begin{align*}
            \Delta_\lambda^\sigma \phi (x)
               & =  \lim_{\varepsilon \to 0^+} \frac{1}{\G(-\sigma)} \int_{|x-y|>\varepsilon} \int_0^\infty
                \big(W_t^\lambda(x,y)-\W_t(x-y)\big)\phi_0(y) \frac{dt\,dy}{t^{1+\sigma}}   \\
                  &\quad + \lim_{\varepsilon \to 0^+} \frac{1}{\G(-\sigma)} \int_{|x-y|>\varepsilon}
                  \int_0^\infty  \W_t(x-y)  \big(\phi_0(y)-\phi(x) \big) \frac{dt\,dy}{t^{1+\sigma}} \\
               & =  \lim_{\varepsilon \to 0^+} \frac{1}{\G(-\sigma)} \int_{|x-y|>\varepsilon}
               \int_0^\infty \big(W_t^\lambda(x,y)-\W_t(x-y)\big)\big(\phi_0(y)-\phi(x)\big)\frac{dt\,dy}{t^{1+\sigma}}  \\
                  & \quad+ \lim_{\varepsilon \to 0^+} \frac{1}{\G(-\sigma)} \int_{|x-y|>\varepsilon}
                  \int_0^\infty \W_t(x-y)  \big(\phi_0(y)-\phi(x)\big) \frac{dt\,dy}{t^{1+\sigma}}  \\
                  & \quad+ \phi(x)\lim_{\varepsilon \to 0^+} \frac{1}{\G(-\sigma)} \int_{|x-y|>\varepsilon}
                  \int_0^\infty  \big(W_t^\lambda(x,y)-\W_t(x-y)\big)\frac{dt\,dy}{t^{1+\sigma}}  \\
               & =  \lim_{\varepsilon \to 0^+} \int_{0, \ |x-y|>\varepsilon}^\infty  \big(\phi(x) - \phi(y) \big)
               K_\sigma^\lambda(x,y)\,dy + \phi(x) B_\sigma^\lambda(x).
        \end{align*}
        In the last identity we use that the double integral defining $B_\sigma^\lambda(x)$ is absolutely convergent
        (see the proof of Lemma \ref{Lem:funcion}).
    \end{proof}

    \begin{Rem}\label{Cor_absint}
    {\rm By using Taylor's formula we can replace the principal value in the integral of \eqref{19.1} by an absolutely convergent integral.
    For $1/2 \leq \sigma < 1$ and $\phi \in \mathcal{S}_\lambda$ we can write
        \begin{align*}
            \Delta_\lambda^\sigma \phi (x)
                &=\int_0^\infty \big(\phi(x)-\phi(y)-\phi'(x)(x-y)\chi_{\{|x-y|\leq 1\}}(x,y)\big)K_\sigma^\lambda(x,y)\, dy \nonumber \\
                &\qquad + \phi(x)B_\sigma^\lambda(x) + \phi'(x)C_\sigma^\lambda(x), \quad x \in \R_+,~\lambda>0,
        \end{align*}
        where $K_\sigma^\lambda(x,y)$ and $B_\sigma^\lambda(x)$ are as in Theorem \ref{Thm:Puntual} and
        $$C_\sigma^\lambda(x)
            = \lim_{\varepsilon\to 0^+}\int_0^\infty  K_\sigma^\lambda(x,y) (x-y) \chi_{\{\varepsilon\leq |x-y|\leq1\}}(x,y)\,dy.$$
      As in \eqref{20.1}, it can be checked (see the proof in Section \ref{sec:demostraciones}) that

\begin{equation}\label{csigma}
|C_{1/2}^\lambda(x)|\leq C\big(1+\chi_{(0,1)}(x)|\ln x|\big),\qquad\hbox{and}\qquad
|C_{\sigma}^\lambda(x)|\leq C x^{1-2\sigma},~\hbox{for}~1/2 < \sigma <1.
\end{equation}}

 \end{Rem}

    We would like to define $\Delta_\lambda^\sigma$ in the space of distributions $\mathcal{S}_\lambda'$, but this is not
    possible because, as we show in Lemma \ref{Lem SS} below,
    the operator $\Delta_\lambda^\sigma$ does not preserve the class $\mathcal{S}_\lambda$. What we can prove
    is that $\Delta_\lambda^\sigma:\mathcal{S}_\lambda\to\mathfrak{S}_\sigma$, where the space
    $\mathfrak{S}_\sigma$ is constituted by all those continuous functions $\varphi$ on $\R_+$ such that
    $$\gamma^\sigma(\varphi) = \sup_{x \in\R_+} (1+x)^{1+2 \sigma}|\varphi(x)| <\infty.$$
    The dual space of the normed space $(\mathfrak{S}_\sigma,\gamma^\sigma)$ is denoted by $\mathfrak{S}_\sigma'$.

    \begin{Lem}\label{Lem SS}
        Let $0<\sigma<1$. Assume that $\lambda \geq1$ when $0<\sigma<1/2$ and that $\lambda\geq2$ when $1/2\leq\sigma<1$.
        Then $\Delta_\lambda^\sigma$ is a bounded operator from $\mathcal{S}_\lambda$ into $\mathfrak{S}_\sigma$.
    \end{Lem}

    \begin{proof}
        We just show the proof for $0<\sigma<1/2$. The case $1/2\leq\sigma<1$ is completely analogous
        in view of Remark \ref{Cor_absint}. Let $\phi \in \mathcal{S}_\lambda$. We need to estimate
        the two terms in the formula for $\Delta_\lambda^\sigma\phi$ of Theorem \ref{Thm:Puntual}
        (see Proposition \ref{Prop_representacion}).
        For the multiplicative term, by \eqref{20.1} it follows that
        $$|\phi(x)B_\sigma^\lambda(x)|
            \leq C x^{-2\sigma} |\phi(x)|
            \leq C \frac{\gamma_{m,0}(\phi)}{(1+x)^{1+2\sigma}},$$
        for some $m \in \mathbb{N}$ large enough. Here we have taken into account that $\lambda\geq1>2\sigma$.
        It remains to consider the integral term. By applying \eqref{estimacion nucleo} it is enough to see that
        \begin{equation}\label{to prove}
        \int_0^\infty \frac{|\phi(x)-\phi(y)|}{|x-y|^{1+2\sigma}}\,dy\leq\frac{C}{(1+x)^{1+2\sigma}}.
        \end{equation}
        Let us distinguish two cases. The first one is when $0<x \leq 2$. Then, since $\phi'$ is bounded
        (because $\lambda\geq1$, see Remark \ref{rem:test acotadas}) and $\phi$ is bounded,
        we can use the mean value theorem to bound the integral in \eqref{to prove} by
        $$\int_{0,\,|x-y|\leq 1}^\infty \frac{\|\phi'\|_{L^\infty(\R_+)}|x-y|}{|x-y|^{1+2\sigma}}\,dy
        +\int_{0, \, |x-y| \geq 1}^\infty \frac{2 \|\phi\|_{L^\infty(\mathbb{R}_+)}}{|x-y|^{1+2\sigma}}\,dy
        \leq C\leq\frac{C}{(1+x)^{1+2\sigma}}.$$
Suppose now that $x >2$.  If $|x-y|<1$, then $y>x/2$. Observe that if in this case we also have $y<x$ then, by the mean value theorem
and reasoning as in Remark \ref{rem:test acotadas},
\begin{align*}
|\phi(x)-\phi(y)|&\leq|x-y| \sup_{x/2\leq\xi\leq x}|\phi'(\xi)|\leq 2^{1+2\sigma}\frac{|x-y|}{x^{1+2\sigma}}
 \sup_{x/2\leq\xi\leq x}\xi^{1+2\sigma}|\phi'(\xi)|\\
&\leq  C\frac{|x-y|}{x^{1+2\sigma}}\big(\gamma_{m,0}(\phi)+\gamma_{m,1}(\phi)\big),
\end{align*}
for a sufficiently large $m$;  while if $y\geq x$ then, as above,
\begin{align*}
|\phi(x)-\phi(y)|&\leq|x-y| \sup_{x\leq\xi\leq y}|\phi'(\xi)|\leq\frac{|x-y|}{x^{1+2\sigma}}
 \sup_{x\leq\xi\leq y}\xi^{1+2\sigma}|\phi'(\xi)|\leq  C\frac{|x-y|}{x^{1+2\sigma}}.
\end{align*}
By using these two estimates we see that, when $x>2$,
        $$\int_{|x-y|\leq1}\frac{|\phi(x)-\phi(y)|}{|x-y|^{1+2\sigma}}\, dy
        \leq  \frac{C}{x^{1+2\sigma}}\int_{|x-y|\leq1}\frac{dy}{|x-y|^{2\sigma}}
 =\frac{C}{x^{1+2\sigma}}\leq \frac{C}{(1+x)^{1+2\sigma}}.$$
On the other hand, since for $y<x/2$ we have $x-y> x/2$,  we can estimate
        \begin{align*}
            \int_{|x-y|>1}&\frac{|\phi(x)-\phi(y)|}{|x-y|^{1+2\sigma}}\,dy\leq
            |\phi(x)|\int_{|x-y|>1}\frac{1}{|x-y|^{1+2\sigma}}\,dy+\int_{|x-y|>1}\frac{|\phi(y)|}{|x-y|^{1+2\sigma}}\,dy \\
            &=C|\phi(x)|+\Bigg( \int_{0, \, |x-y|>1}^{x/2}~+\int_{x/2, \, |x-y|>1}^\infty  ~\Bigg)
            \frac{|\phi(y)|}{|x-y|^{1+2\sigma}}\,dy\\
            &\leq C\frac{x^{1+2\sigma+\lambda}|x^{-\lambda}\phi(x)|}{x^{1+2\sigma}}+\frac{C}{x^{1+2\sigma}}
            \int_{0}^{x/2}|\phi(y)|\,dy + C \int_{x/2, \, |x-y|>1}^\infty \frac{y^{1+2\sigma+\lambda}|y^{-\lambda}\phi(y)|}
            {y^{1+2\sigma}|x-y|^{1+2\sigma}}\,dy\\
            &\leq \frac{C}{x^{1+2\sigma}}\big(\gamma_{m,0}(\phi)+\|\phi\|_{L^1(\R_+)}\big) +
            C\frac{\gamma_{m,0}(\phi)}{x^{1+2\sigma}}\int_{|x-y|>1,\,y\in\R}\frac{1}{|x-y|^{1+2\sigma}}\,dy\\
&\leq \frac{C}{(1+x)^{1+2\sigma}},
        \end{align*}
as soon as $m \in \N$ and $m\geq 1+2\sigma+\lambda$. Hence \eqref{to prove} holds.
    \end{proof}

    The operator $\Delta_\lambda^\sigma$ is defined
    in the space $\mathfrak{S}_\sigma'$ by transposition, that is, if $T \in \mathfrak{S}_\sigma'$ then
    $\Delta_\lambda^\sigma T$ is the element of $\mathcal{S}_\lambda'$ defined by
    $$ \Delta_\lambda^\sigma T( \phi)
        = T( \Delta_\lambda^\sigma \phi), \quad  \phi \in\mathcal{S}_\lambda.$$
    We do not need the most abstract and general definition of $\Delta_\lambda^\sigma$ because we work with H\"older continuous functions.
    Consider the space $L_\sigma$ as defined in \eqref{Lrho}.
    If $u \in L_\sigma$, then clearly $u$ defines an element $T_u$ of $\mathfrak{S}_\sigma'$ by
    $$T_u( \psi)
        = \int_0^\infty u(x) \psi(x)dx, \quad \psi \in \mathfrak{S}_\sigma.$$
    To simplify the notation we make the identification $u\equiv T_u$.

    \begin{proof}[Proof of Theorem \ref{Thm:Puntual}]
       Suppose that $u \in C^\alpha_+$, for $0<2\sigma<\alpha< 1$, or that $u\in\mathrm{Lip}_+$ when $\alpha=1$ . For every $j \in \N$, we define the function
       $$\phi_j(x)=\xi_j(x)(\W_{1/j}*u_o)(x), \quad x \in \R_+,$$
       where $u_o$ represents the odd extension of $u$ to the real line, and $\xi_j \in C_c^\infty(\R_+)$ is
       such that $0 \leq \xi_j \leq 1$ and
       $$\xi_j(x)
            = \left\{
                \begin{array}{ll}
                    1, & x \in [2^{-j},j], \\
                    0, & x \in \R_+ \setminus (2^{-j-1},j+1).
                \end{array} \right.$$
       Since $u \in L_\sigma$ we have that
       \begin{align*}
            |\W_{1/j}*u_o(x)|
                &\leq  C\sqrt{j} \int_0^\infty  e^{-j|x-z|^2/4} |u(z)| \,dz + C \sqrt{j} \int_0^\infty  e^{-j(x+z)^2/4} |u(z)| \,dz \\
                &\leq  C \sqrt{j} \int_0^{2x} |u(z)| \,dz + C \sqrt{j} \int_0^\infty  e^{-cjz^2} |u(z)|\, dz \\
                &\leq  C (\sqrt{j} (1+x)^{1+2\sigma} + j^{-\sigma}) \|u\|_{L_\sigma}, \quad x \in \R_+, \ j \in \N.
       \end{align*}
       Moreover, it is clear that $\phi_j \in C_c^\infty(\R_+) \subset \mathcal{S}_\lambda$, $j \in \N$.
       The family $\{\phi_j\}_{j \in \N}$ converges uniformly to $u$ on compact sets of $\R_+$
       as $j\to\infty$, and it also converges to $u$ in $L_\sigma$.
       In fact, if $||| \cdot |||$ denotes the supremum over a compact set of $\R_+$ or the norm in $L_\sigma$, then we have
       \begin{equation}\label{pera}
       \begin{aligned}
            ||| \phi_j - u |||
                & = C ||| \int_\R e^{-z^2/4} [\xi_j (\cdot)u_o(\cdot - z/\sqrt{j}) - u_o(\cdot)]\, dz||| \\
                & \leq C \int_\R e^{-z^2/4} ||| \xi_j (\cdot) (u_o(\cdot - z/\sqrt{j}) - u_o(\cdot)) + (\xi_j(\cdot)-1)u(\cdot) ||| \,dz \\
                & \leq C j^{-\alpha/2} \int_\R e^{-z^2/4} z^\alpha\, dz + C |||(\xi_j(\cdot)-1)u(\cdot) ||| \to0, \quad \text{as } j \to \infty,
       \end{aligned}
       \end{equation}
       because $u_o$ is in $C^\alpha(\R)$ (or in $\mathrm{Lip}(\R)$) (see Lemma~\ref{Lem oddHolder} below).
       Moreover, as a consequence of Lemma~\ref{Lem SS} and by transposition
       \begin{equation}\label{converg S'}
            \Delta_\lambda^\sigma \phi_j \to \Delta_\lambda^\sigma u, \quad \text{as } j \to \infty, \text{ in } \mathcal{S}_\lambda'.
       \end{equation}
       Our next aim is to show that
       \begin{equation}\label{converg unif}
            \int_0^\infty (\phi_j(x)-\phi_j(y))K_\sigma^\lambda(x,y)\,dy
                \to \int_0^\infty (u(x)-u(y))K_\sigma^\lambda(x,y)\,dy, \quad \text{as } j \to \infty,
       \end{equation}
       uniformly on compact sets of $\R_+$. Indeed, let $Q$ be a compact subset of $\R_+$ and $\varepsilon>0$.
       There exist $\ell \in \N$ and $C>1$,
       such that $Q \subset [2^{-\ell},\ell]$ and, for every $j \geq \ell +1$,
       $$\frac{|\phi_j(x)-\phi_j(y)|}{|x-y|^\alpha} \leq C \|u\|_{C^\alpha_+} =:M, \quad \hbox{for all}~x \in Q, \ |x-y|<2^{-\ell-1}. $$
       Now, according to \eqref{estimacion nucleo} we fix $0<\delta<2^{-\ell-1}$ verifying
       $$M \int_{|x-y|<\delta} |x-y|^\alpha K_\sigma^\lambda(x,y)\, dy < \frac{\varepsilon}{3}, \quad
\hbox{for all}~ x \in Q.$$
       Then, for every $j \geq \ell +1$, we have that
       $$\int_{|x-y|<\delta} |\phi_j(x)-\phi_j(y)| K_\sigma^\lambda(x,y)\,dy + \int_{|x-y|<\delta} |u(x)-u(y)| K_\sigma^\lambda(x,y)\,dy
            \leq \frac{2\varepsilon}{3}, \quad x \in Q.$$
       On the other hand, by taking into account again \eqref{estimacion nucleo}, we get
       \begin{align*}
            & \int_{|x-y| \geq \delta} |\phi_j(x)-u(x)| K_\sigma^\lambda(x,y)\,dy
                + \Big(\int_{0, \ |x-y| \geq \delta}^{2x} + \int_{2x, \ |x-y| \geq \delta}^{\infty}\Big) |\phi_j(y)-u(y)| K_\sigma^\lambda(x,y)\,dy \\
            & \qquad \qquad \leq C_\delta \sup_{x \in Q} |\phi_j(x)-u(x)|
                + C_\delta (1+x)^{1+2\sigma} \|\phi_j-u\|_{L_\sigma}
                + C_\delta \|\phi_j-u\|_{L_\sigma}<\frac{\varepsilon}{3},
       \end{align*}
        for all $x \in Q$,  provided that $j$ is large enough. Thus \eqref{converg unif} is proved.

       Finally, from estimate \eqref{20.1},
       $\phi_j B_\sigma^\lambda$ converges uniformly to $u B_\sigma^\lambda$
       on compact sets of $\R_+$ as $j\to\infty$.
       On the other hand, by proceeding as above, we can obtain
       \begin{align*}
            & \Big| \int_0^\infty (u(x)-u(y)) K_\sigma^\lambda(x,y)\, dy + u(x) B_\sigma^\lambda(x) \Big| \\
            & \qquad \leq C \Bigg[ \int_{|x-y| \leq 1} \frac{dy}{|x-y|^{1+2\sigma-\alpha}}
                           + \Big(\int_{0, |x-y|>1}^{2x} + \int_{2x, |x-y|>1}^{\infty} \Big) \frac{|u(x)-u(y)|}{|x-y|^{1+2\sigma}} \,dy
                           + \frac{|u(x)|}{x^{2\sigma}}  \Bigg] \\
            & \qquad \leq C \Big( 1 + x^\alpha + (1+x)^{1+2\sigma} + \int_{2x}^\infty \frac{|u(y)|}{(1+y)^{1+2\sigma}}\,dy + x^{\alpha-2\sigma} \Big) \\
            & \qquad \leq C (1+x)^2, \quad x \in \R_+.
       \end{align*}
       Therefore, the function
       $$\mathcal{U}(x)
            = \int_0^\infty (u(x)-u(y)) K_\sigma^\lambda(x,y)\,dy + u(x)B_\sigma^\lambda(x), \quad x \in \R_+, $$
       defines an element of $\mathcal{S}_\lambda'$.
       Hence, by the uniqueness of the limits, \eqref{converg S'}  and \eqref{converg unif} together with
       Proposition~\ref{Prop_representacion}, we conclude that  the pointwise formula \eqref{10.1}
       holds for $u\in C^\alpha_+\cap L_\sigma$ (or $u\in \mathrm{Lip}_+\cap L_\sigma$). Furthermore, $\Delta_\lambda^\sigma u$ is a continuous function
       because it is the uniform limit in any compact subset of $\mathbb{R}_+$ of continuous functions.

       Assume now that $u' \in C^\alpha_+$, with $0 \leq 2\sigma - 1 < \alpha < 1$,  or $u'\in \mathrm{Lip}_+$ when $\alpha=1$. Then we can proceed as in the previous case.
       In fact, being $\W_{1/j}(z)z$ an odd function, $j \in \N$,
       we can write the sequence $\{\phi_j\}_{j \in \N}$ above as
       $$\phi_j(x)
            = \xi_j(x) \int_\R \W_{1/j}(x-z) (u_o(z)+ u'(x)(x-z)) \,dz, \quad x \in \mathbb{R}_+.$$
       By  the mean value theorem and  since $(u')_o$ is in $C^\alpha(\R)$ (or in $\mathrm{Lip}_+$ when $\alpha=1$),
       we can see that this family preserves the same convergence properties as before.
       Further, $\phi_j'$ converges uniformly to $u'$ on compact sets of $\R_+$ as $j\to\infty$.
       By taking into account Remark~\ref{Cor_absint}, the key point is to show that
       \begin{align*}
            & \int_0^\infty (\phi_j(x)-\phi_j(y)-\phi_j'(x)(x-y)\chi_{\{|x-y| \leq 1\}}(x,y))K_\sigma^\lambda(x,y)\,dy \\
            & \qquad \qquad \to\int_0^\infty (u(x)-u(y)-u'(x)(x-y)\chi_{\{|x-y| \leq 1\}}(x,y))K_\sigma^\lambda(x,y)\,dy, \quad \text{as } j \to \infty,
       \end{align*}
       uniformly on compact sets of $\R_+$. Again, fix $\varepsilon>0$ and a compact set $Q \subset [2^{-\ell},\ell]$, for some $\ell \in \N$.
       In this case we have that
       $$\frac{|\phi_j(x)-\phi_j(y)-\phi_j'(x)(x-y)|}{|x-y|^{\alpha+1}} \leq C \|u'\|_{C^\alpha_+} =:M, \quad
        \hbox{for all}~x \in Q, \ |x-y|<2^{-\ell-1},~j > \ell +1.$$
       Hence, we take $0<\delta<2^{-\ell-1}$ such that
       $$M \int_{|x-y|<\delta} |x-y|^{\alpha+1} K_\sigma^\lambda(x,y) dy < \frac{\varepsilon}{3}, \quad x \in Q,$$
       and we repeat the preceding reasoning. Further details are omitted.
    \end{proof}

    \begin{Cor}\label{Thm:Calpha}
        Suppose that $\lambda\geq1$ when $0<\sigma<1/2$ and that $\lambda\geq2$ when $1/2\leq\sigma<1$.
        Let $u\in L_\sigma\cap C_+^\alpha$ when $0<2\sigma<\alpha < 1$ (or $u\in L_\sigma\cap \mathrm{Lip}_+$ if $\alpha=1$);
        or $u\in L_\sigma$, $u$ continuous on $[0,\infty)$ and
        $u'\in C^\alpha_+$ when $0 \leq 2\sigma-1<\alpha <1$ (resp. $u'\in\mathrm{Lip}_+$ if $\alpha=1$). Then,
        $$\Delta_\lambda^\sigma u (x)
            = \frac{1}{\G(-\sigma)}  \int_0^\infty \big(W_t^\lambda u(x)-u(x) \big) \frac{dt}{t^{1+\sigma}}, \quad x\in\R_+.$$
    \end{Cor}

    \begin{proof}
        Assume firstly that $0<\sigma<1/2$. By Theorem~\ref{Thm:Puntual} it is enough to write down the kernel $K^\lambda_\sigma(x,y)$ in terms of $W_t^\lambda(x,y)$ and to show that we can apply
        Fubini's theorem in the first integral of \eqref{10.1}. Indeed, note that
        \begin{align*}
            & \int_0^\infty |u(x)-u(y)| K_\sigma^\lambda(x,y)\,dy
                \leq C \Big( \int_{|x-y|<1} + \int_{0, \ |x-y| >1}^{2x} + \int_{2x, \ |x-y| >1}^{\infty}  \Big)\frac{|u(x)-u(y)|}{|x-y|^{1+2\sigma}} \,dy \\
            & \qquad \qquad \leq C \int_{|x-y|<1} \frac{dy}{|x-y|^{1+2\sigma-\alpha}} +
            C x |u(x)| + C \|u\|_{L_\sigma}    <\infty.
        \end{align*}

        Now, we consider the case $1/2\leq\sigma<1$. From the proof of Theorem~\ref{Thm:Puntual} we deduce that
        \begin{align*}
            \Delta_\lambda^\sigma u (x)
                =&\int_0^\infty \big(u(x)-u(y)-u'(x)(x-y)\chi_{\{|x-y|\leq 1\}}(x,y)\big)K_\sigma^\lambda(x,y)\, dy  \\
                & + u(x)B_\sigma^\lambda(x) + u'(x)C_\sigma^\lambda(x), \quad x \in \R_+.
        \end{align*}
        Once again, by using the mean value theorem, we can see that the integral related to the kernel $K_\sigma^\lambda$ is absolutely convergent.
        Thus, $\Delta_\lambda^\sigma u (x)$ equals to
        \begin{align*}
                & \lim_{\varepsilon \to 0^+} \frac{1}{-\Gamma(-\sigma)} \Big[ \int_0^\infty \int_0^\infty \big(u(x)-u(y)-u'(x)(x-y)\chi_{\{\varepsilon \leq |x-y|\leq 1\}}(x,y)\big)W_t^\lambda(x,y)\, \frac{dy\,dt}{t^{1+\sigma}}  \\
                & - \int_0^\infty \big(u(x)W_t^\lambda 1(x) - u(x)\big) \frac{dt}{t^{1+\sigma}} + \int_0^\infty \int_0^\infty u'(x)(x-y)\chi_{\{\varepsilon \leq |x-y|\leq 1\}}(x,y) W_t^\lambda(x,y)\, \frac{dy\,dt}{t^{1+\sigma}} \Big] \\
                &=  \frac{1}{\G(-\sigma)}  \int_0^\infty \big(W_t^\lambda u(x)-u(x) \big) \frac{dt}{t^{1+\sigma}}.
        \end{align*}
        Note that the interchange in the order of integration in the integral defined by $C_\sigma^\lambda$ is legitimate
        for every $\varepsilon>0$.
    \end{proof}

Next we pass to the study of the pointwise limits for the fractional Bessel operator.
By \eqref{spectral} and \eqref{def frac}, it is clear that if $\lambda>0$ and $\phi \in \mathcal{S}_\lambda$, then for all $x \in \R_+$,
    $$\lim_{\sigma \to 0^+} \Delta_\lambda^\sigma \phi(x)= \phi(x)
    \quad \text{and} \quad
    \lim_{\sigma \to 1^-} \Delta_\lambda^\sigma \phi(x)= \Delta_\lambda \phi(x).$$
    We now extend these properties to a larger class of functions.

    \begin{Cor}\label{Prop sigma1}
Let $u\in L^\infty(\R_+)$.
        \begin{itemize}
            \item[$(i)$] If $\lambda \geq 1$ and $u \in C^\alpha_+\cap L_0$, for some $0<\alpha<1$, then
            $$\lim_{\sigma \to 0^+} \Delta_\lambda^\sigma u(x)
                = u(x), \quad\hbox{for all}~x\in\R_+.$$
            \item[$(ii)$] If $\lambda \geq 2$, $u\in C^2(\R_+)$ and $u' \in L^\infty(\R_+)$, then
            $$\lim_{\sigma \to 1^-} \Delta_\lambda^\sigma u(x)
                = \Delta_\lambda u(x), \quad\hbox{for all}~x\in\R_+.$$
        \end{itemize}
    \end{Cor}

    \begin{proof}
        Suppose that $0< \sigma < 1$. Take $u$ satisfying the hypotheses \textit{(i)} or \textit{(ii)} above and fix $x \in \R_+$.
        Extend $u$ by $0$ on $(-\infty,0)$ and call this new function $u_0$.
        By Corollary~\ref{Thm:Calpha} we can split $\Delta_\lambda^\sigma u (x)$ as
        \begin{align*}
            \Delta_\lambda^\sigma u (x)
                &= \frac{1}{\G(-\sigma)} \int_0^\infty \big(W_t^\lambda u(x)- \W_tu_0(x)\big)\frac{dt}{t^{1+\sigma}}
                    + \frac{1}{\G(-\sigma)} \int_0^\infty \big(\W_tu_0(x)- u(x) \big)\frac{dt}{t^{1+\sigma}} \\
                &=  L_\sigma (u)(x) + \left( - \frac{d^2}{dx^2}\right)^\sigma u_0(x).
        \end{align*}
        According to \cite[Proposition 2.5]{StiPhD} and \cite[Proposition 5.3 and Remark 5.4]{Stinga-Torrea} we have that
        $$\lim_{\sigma \to 0^+} \left( - \frac{d^2}{dx^2}\right)^\sigma u_0(x)= u(x),
        \quad \hbox{or} \quad
        \lim_{\sigma \to 1^-} \left( - \frac{d^2}{dx^2}\right)^\sigma u_0(x)= -u''(x),$$
        depending if we are in case \textit{(i)} or in case \textit{(ii)}, respectively.
        Our next objectives are to show that
        \begin{equation}\label{obj1}
            \lim_{\sigma \to 0^+} L_\sigma u(x) = 0
            \quad \text{and} \quad
            \lim_{\sigma \to 1^-} L_\sigma u(x) = \frac{\lambda(\lambda-1)}{x^2}u(x),
        \end{equation}
        where the first limit is for \textit{(i)} and the second one for \textit{(ii)}.
        We write down the heat kernels and decompose $L_\sigma u(x)$ in three integrals
        $$\frac{1}{\G(-\sigma)}\int_0^\infty \Bigg( \int_0^{x/2} + \int_{x/2}^{3x/2} + \int_{3x/2}^\infty  \Bigg)
           \big(W_t^\lambda(x,y)-\W_t(x-y)\big)u(y)\, \frac{dy\,dt}{t^{1+\sigma}}=\sum_{i=1}^3L_{\sigma,i}(u)(x).$$
        By \eqref{control clasico} and \eqref{cuenta clasica},
        \begin{align}\label{4.10}
            |L_{\sigma,1}(u)(x)|
                &\leq \frac{C}{|\G(-\sigma)|}  \int_0^{x/2} \frac{|u(y)|}{|x-y|^{1+2\sigma}}\,dy
                \leq \frac{C}{|\G(-\sigma)| x^{2\sigma}} \|u\|_{L^\infty(\R_+)},
        \end{align}
        with $C$ independent of $\sigma$. Analogously,
        \begin{align}\label{4.11}
            |L_{\sigma,3}(u)(x)|
                & \leq  \frac{C}{|\G(-\sigma)|} \int_{3x/2}^\infty \frac{|u(y)|}{|x-y|^{1+2\sigma}}\, dy
                  \leq \frac{C}{|\G(-\sigma)|} \left(1 + \frac{1}{x} \right)^{1+2\sigma} \|u\|_{L_\sigma}.
        \end{align}
        We also split $L_{\sigma,2}$ as follows
        \begin{align*}
            L_{\sigma,2}(u)(x)
                &= \frac{1}{\G(-\sigma)} \int_{x/2}^{3x/2} \left(\int_{xy}^\infty + \int_0^{xy} \right)
                \left[ W_t^\lambda(x,y)- \W_t(x-y) \right]u(y) \,\frac{dt\,dy}{t^{1+\sigma}}\\
                &=  L_{\sigma,2,1}(u)(x) + L_{\sigma,2,2}(u)(x).
        \end{align*}
        By using again \eqref{control clasico} we obtain
        \begin{align}\label{4.12}
            \left| L_{\sigma,2,1}(u)(x) \right|
                &\leq \frac{C}{|\G(-\sigma)|} \int_{x/2}^{3x/2} \int_{xy}^\infty   \frac{e^{-|x-y|^2/(4t)}}{t^{\sigma+3/2}} |u(y)|\,dt\,dy
                \leq  \frac{C}{|\G(-\sigma)|} \int_{x/2}^{3x/2}  \frac{|u(y)|}{x^{2\sigma+1}} \,dy \nonumber \\
                &\leq \frac{C}{|\G(-\sigma)| x^{2\sigma}} \|u\|_{L^\infty(\R_+)}.
        \end{align}
         From \eqref{4.10}, \eqref{4.11} and \eqref{4.12} we deduce that, under the sole
assumption that $u\in L^\infty(\R_+)\cap L_0\subset L_\rho$, $\rho>0$,
        $$\lim_{\sigma \to 0^+, 1^-} \left[ L_{\sigma,1}(u)(x) + L_{\sigma,3}(u)(x) + L_{\sigma,2,1}(u)(x)\right]
            = 0,$$
        because $1/\Gamma(-\sigma)\sim \sigma(1-\sigma)$ when $0<\sigma<1$.

        In order to analyze $L_{\sigma,2,2}$ we take into account \eqref{25.1}. Observe that, if $0<\sigma < 1/2$, then
        \begin{align*}
            \left| L_{\sigma,2,2}(u)(x) \right|
                \leq & \frac{C}{|\G(-\sigma)|} \int_{x/2}^{3x/2} \int_0^{xy} \W_t(x-y) \frac{t}{xy} |u(y)|\,\frac{dt\,dy}{t^{1+\sigma}} \\
                \leq & \frac{C}{|\G(-\sigma)|} \int_{x/2}^{3x/2} \frac{1}{x^2} \int_0^{xy}  \frac{dt}{t^{1/2+\sigma}} |u(y)|\, dy
                \leq  \frac{C}{|\G(-\sigma)| x^{2\sigma}} \|u\|_{L^\infty(\R_+)}.
        \end{align*}
        Hence, $L_{\sigma,2,2}(u)(x) \to 0$, as $\sigma \to 0^+$, and the proof of $(i)$ is completed.

        On the other hand, to complete the proof of $(ii)$ we need to treat the operator $L_{\sigma,2,2}$ more carefully
        by considering the precise asymptotics in \eqref{25.1}. Suppose that $1/2 < \sigma < 1$. We can estimate
        \begin{align*}
            \frac{1}{|\G(-\sigma)|} \int_{x/2}^{3x/2} \int_0^{xy} \W_t(x-y) \frac{t^2}{(xy)^2}\, \frac{dt\,dy}{t^{1+\sigma}}
              &  \leq  \frac{C}{|\G(-\sigma)|} \int_{x/2}^{3x/2} \frac{1}{x^4} \int_0^{xy} t^{1/2-\sigma}\, dt\, dy\nonumber\\
&             \leq  \frac{C}{|\G(-\sigma)| x^{2\sigma}}.
        \end{align*}
        We also have that
        \begin{align*}
            & - \frac{\lambda(\lambda-1)}{\G(-\sigma)} \int_{x/2}^{3x/2} \int_0^{xy} \W_t(x-y) \frac{t}{xy} \,u(y)\,\frac{dt\,dy}{t^{1+\sigma}} \nonumber \\
            & \qquad = -\frac{\lambda(\lambda-1)}{\sqrt{4\pi}\G(-\sigma)} \int_{x/2}^{3x/2} \frac{1}{xy}
                    \left(\int_0^\infty \frac{e^{-|x-y|^2/(4t)}}{t^{\sigma+1/2}} \,dt -
                    \int_{xy}^\infty \frac{e^{-|x-y|^2/(4t)}}{t^{\sigma+1/2}}\, dt  \right) u(y)\, dy \nonumber\\
            & \qquad = -\frac{\lambda(\lambda-1)}{\sqrt{4\pi}\G(-\sigma)} \int_{x/2}^{3x/2} \frac{1}{xy}
                    \left( \frac{4^{\sigma-1/2}\G(\sigma-1/2)}{|x-y|^{2\sigma-1}} -
                    \int_{xy}^\infty \frac{e^{-|x-y|^2/(4t)}}{t^{\sigma+1/2}}\, dt  \right) u(y)\, dy.
        \end{align*}
        Since,
        $$ \frac{1}{|\G(-\sigma)|} \int_{x/2}^{3x/2} \frac{1}{xy} \int_{xy}^\infty \frac{e^{-|x-y|^2/(4t)}}{t^{\sigma+1/2}} \,dt\, dy
            \leq \frac{C}{|\G(-\sigma)|x^{2\sigma}},$$
        and $\lim_{\sigma\to1^-}(4\pi)^{-1/2}4^{\sigma-1/2}\Gamma(\sigma-1/2)=1$, we see that we have reduced \eqref{obj1}, for $\sigma \to 1^-$, to showing that
        \begin{equation*}
            \lim_{\sigma \to 1^-}
                \frac{1}{|\G(-\sigma)|} \int_{x/2}^{3x/2} \frac{1}{y}\frac{u(y)}{|x-y|^{2\sigma-1}}  \, dy
                    = \frac{u(x)}{x}.
        \end{equation*}
        Now we write the integral above as
        \begin{align*}
             &\frac{1}{|\G(-\sigma)|} \int_{x/2}^{3x/2} \Big(\frac{1}{y} - \frac{1}{x} \Big)\frac{u(y)}{|x-y|^{2\sigma-1}}\,dy   + \frac{1}{|\G(-\sigma)|} \int_{x/2}^{3x/2} \frac{1}{x}\frac{u(y)-u(x)}{|x-y|^{2\sigma-1}}\,dy\\
&               \qquad   + \frac{1}{|\G(-\sigma)|} \int_{x/2}^{3x/2} \frac{1}{x}\frac{u(x)}{|x-y|^{2\sigma-1}}\,dy.
        \end{align*}
        The first two terms above converge to $0$ as $\sigma \to 1^-$, because
        $$ \int_{x/2}^{3x/2} \left| \frac{1}{y} - \frac{1}{x} \right| \frac{dy}{|x-y|^{2\sigma-1}}
            \leq \frac{C}{x^2} \int_{x/2}^{3x/2}  \frac{dy}{|x-y|^{2\sigma-2}}
            \leq \frac{C}{x^{2\sigma-1}},$$
        and
        $$ \frac{1}{x}\int_{x/2}^{3x/2}  \frac{|u(y)-u(x)|}{|x-y|^{2\sigma-1}}\,dy
            \leq C\frac{\|u'\|_{L^\infty(\R_+)}}{x} \int_{x/2}^{3x/2}  \frac{dy}{|x-y|^{2\sigma-2}}
            \leq \frac{C}{x^{2\sigma-2}}.$$
        Finally, since $\Gamma(z+1)=z\Gamma(z)$,
        \begin{align*}
            \lim_{\sigma \to 1^-} \frac{1}{|\G(-\sigma)|} \int_{x/2}^{3x/2} \frac{1}{x}\frac{u(x)}{|x-y|^{2\sigma-1}}\,dy
                = \frac{u(x)}{x } \lim_{\sigma \to 1^-} \frac{1}{|\G(-\sigma)| (1-\sigma) } \left( \frac{x}{2} \right)^{2-2\sigma}
                = \frac{u(x)}{x}.
        \end{align*}
    \end{proof}

     \section{The $\Delta_\lambda$-harmonic extension for functions in $C^\alpha_+$}
\label{sec:Poisson}

\begin{Prop}\label{thm:harmonic extension}
Let $\lambda>0$ and $f\in L_{\lambda/2}$. Assume also that $f$ is in $C^\alpha_+$ for some $0<\alpha<1$.
Then the Bessel harmonic extension of $f$, namely, the solution to
    \eqref{Bessel harmonic extension} is given by
   \begin{equation}\label{subord}
    v(x,t)\equiv P_t^\lambda f(x)=\frac{t}{2\sqrt{\pi}} \int_0^\infty \frac{e^{-t^2/(4s)}}{s^{3/2}} W_s^\lambda f(x)\,ds.
    \end{equation}
    If we define the Poisson kernel for $x,y\in\R_+$, $t>0$, as
 \begin{equation}\label{subord kernel}
    P_t^\lambda(x,y)
            = \frac{t}{2\sqrt{\pi}}\int_0^\infty \frac{e^{-t^2/(4s)}}{s^{3/2}} W_s^\lambda(x,y)\,ds
 =\frac{1}{\sqrt{\pi}}\int_0^\infty \frac{e^{-r}}{r^{1/2}}W_{t^2/(4r)}^\lambda(x,y)\,dr,
    \end{equation}
then we can write the classical solution to \eqref{Bessel harmonic extension} as
\begin{equation}\label{definicion de v}
v(x,t)=\int_{\R_+}P_t^\lambda(x,y)f(y)\,dy,\quad x\in\R_+,~t>0.
\end{equation}
\end{Prop}

{Let us introduce the concept of fractional derivative as defined by C. Segovia and R. L. Wheeden
    in \cite{SW}.
    Let $\beta>0$. We choose $m \in \N$ such that $m-1 \leq \beta <m$. Suppose that $F(x,t)$
    is a nice enough function on $\R_+^2$. Then, we define \cite[p.~248]{SW}
    \begin{equation}\label{fractional derivative}
    \partial_t^\beta F(x,t)
        = \frac{e^{-i\pi(m-\beta)}}{\G(m-\beta)} \int_0^\infty \partial_t^m F(x,t+s)s^{m-\beta-1}ds,\quad x\in\R_+,t>0.
     \end{equation}
    This fractional derivative was used in \cite{SW} by Segovia and Wheeden to
    define Littlewood--Paley functions that characterize fractional Sobolev spaces.
    In other contexts, the $\partial_t^\beta$  operator is useful to prove that potential
    spaces associated to orthogonal
    expansions coincide with the corresponding Sobolev spaces, even in
    the vector valued situation, see for example
    \cite{BCCFR1, Jorge, Sanabria}. Also, in \cite{MSTZ1} this fractional derivative was used to
    characterize H\"older spaces in the Schr\"odinger setting.}

Before proving this result we collect some
estimates for the Poisson kernel \eqref{subord kernel} that will be useful later.

    \begin{Lem}\label{Lem fractPoiss}
        Let $\lambda,\beta>0$. Recall that $\widetilde{\lambda}=\min\{\lambda,1\}$. Then, for all $x,y\in\R_+$ and $t>0$,
        \begin{equation}\label{F0}
             P_t^\lambda(x,y)
                \leq C \frac{t}{(t+|x-y|)^2} \left( \frac{x \wedge y}{t+|x-y|} \wedge 1 \right)^\lambda,
        \end{equation}
        \begin{equation}\label{F0''}
             P_t^\lambda(x,y)
                \leq C \frac{t(xy)^\lambda}{(t+|x-y|)^{2(\lambda+1)}},
        \end{equation}
        \begin{equation}\label{F1}
            \left| \partial_t^\beta P_t^\lambda(x,y) \right|
                \leq C \frac{y^{\tlambda}}{(t+|x-y|)^{\tlambda+\beta+1}},
        \end{equation}
        \begin{equation}\label{F1'}
            \left| \partial_t^\beta P_t^\lambda(x,y) \right|
                \leq C \frac{y^{\lambda}}{(t+|x-y|)^{\lambda+\beta+1}},
        \end{equation}
        \begin{equation}\label{F1''}
            \left| \partial_t^\beta P_t^\lambda(x,y) \right|
                \leq C \frac{1}{(t+|x-y|)^{\beta+1}},
        \end{equation}
        \begin{equation}\label{F0'}
            \left| \partial_t^\beta P_t^\lambda(x,y) \right|
                \leq C \frac{1}{(t+|x-y|)^{\beta+1}} \left( \frac{x \wedge y}{t+|x-y|} \wedge 1 \right)^\lambda,
        \end{equation}
        and
        \begin{equation}\label{F2}
            \left| \partial_y \partial_t^\beta P_t^\lambda(x,y) \right|
                \leq C \left(\frac{y^{\tlambda-1}}{(t+|x-y|)^{\tlambda+\beta+1}} + \frac{1}{(t+|x-y|)^{\beta+2}}\right).
        \end{equation}
    \end{Lem}

    \begin{proof}
    We show how to prove \eqref{F1} and \eqref{F2} in detail. The rest of the estimates will follow in a similar way
    and the computations are left to the interested reader.
    We need the following estimate that was proved in \cite[Lemma 3]{BCCFR1}:
    for any $\beta>0$, there exists a constant $C$ such that for all $t,u>0$,
    \begin{equation}\label{beta lema}
    |\partial_t^\beta (te^{-t^2/(4u)})|\leq Ce^{-t^2/(8u)}u^{(1-\beta)/2}.
    \end{equation}
    Let $m$ be an integer with $m-1 \leq \beta < m$.
    By \eqref{subord kernel} and the definition of $\partial_t^\beta$ in \eqref{fractional derivative} we have,
    for any $x,y\in\R_+$ and $t>0$,
        \begin{align*}
            \partial_t^\beta P_t^\lambda(x,y)
                & =  C\int_0^\infty \int_0^\infty
                    \partial_t^m\left[ (t+s)e^{-(t+s)^2/(4u)} \right] W_u^\lambda(x,y)\frac{ du }{{u^{3/2}} }s^{m-\beta-1}ds\\
                & =  C\int_0^\infty
                    \partial_t^\beta\left[ te^{-t^2/(4u)} \right] W_u^\lambda(x,y)\frac{ du}{u^{3/2}}.
        \end{align*}
    The differentiation under the integral sign and the interchange of the order of integration is correct because,
    according to \eqref{beta lema} and \eqref{control clasico}, we have that
        \begin{align*}
            \int_0^\infty &\int_0^\infty  \left|\partial_t^m\left[ (t+s)e^{-(t+s)^2/4u} \right]\right|
            W_u^\lambda(x,y) \,\frac{du}{u^{3/2}} s^{m-\beta-1}\,ds \\
               & \leq  C \int_0^\infty s^{m-\beta-1} \int_0^\infty  \frac{e^{-[(t+s)^2+|x-y|^2]/8u}}{u^{(m+3)/2}}\, du \,ds
                \leq  C \int_0^\infty \frac{s^{m-\beta-1}}{[(t+s)^2+|x-y|^2]^{(m+1)/2}} \,ds \\
                &\leq  C \frac{t^{-m}+1}{t+|x-y|}, \quad x,y \in \R_+, \ t>0.
        \end{align*}

        Using \eqref{beta lema}, \eqref{W origen} and \eqref{W infinito},
\begin{equation}\label{asterisco}
\begin{aligned}
         \left| \partial_t^\beta P_t^\lambda(x,y) \right|
                &\leq   C \int_0^\infty \frac{e^{-t^2/(8u)}}{u^{(\beta+2)/2}} W_u^\lambda(x,y)\, du \\
                &\leq  C \left( \int_0^{xy/2} \frac{e^{-(t^2+|x-y|^2)/(8u)}}{u^{(\beta+3)/2}}  du
                            + \int_{xy/2}^\infty \frac{e^{-(t^2+x^2+y^2)/(8u)}}{u^{(\beta+3)/2}} \left( \frac{xy}{u} \right)^\tlambda  du \right).
\end{aligned}
\end{equation}
        It is not hard to see that the second integral above is bounded by
        \begin{equation}\label{first}
         \frac{(xy)^\tlambda}{(t^2+x^2+y^2)^{\tlambda + (\beta+1)/2}}
            \leq C \frac{y^\tlambda}{(t +|x-y|)^{\tlambda + \beta+1}}.
            \end{equation}
       As for the first one, if $x<2y$ then the bound we get is
\begin{equation}\label{asterisco2}
        \begin{aligned}
            \int_0^{xy/2} \frac{e^{-(t^2+|x-y|^2)/8u}}{u^{(\beta+3)/2}}  \left( \frac{xy}{u} \right)^{\tlambda/2} du
                &\leq C \frac{(xy)^{\tlambda/2}}{(t^2+|x-y|^2)^{(\tlambda + \beta+1)/2}} \\
                &\leq C \frac{y^\tlambda}{(t +|x-y|)^{\tlambda + \beta+1}},
        \end{aligned}
\end{equation}
        while when $x\geq 2y$ the bound becomes
\begin{equation}\label{asterisco3}
        \begin{aligned}
            \int_0^{xy/2} \frac{e^{-(t^2+|x-y|^2)/8u}}{u^{(\beta+3)/2}}&  \left( \frac{xy}{u} \right)^{\tlambda} du
                \leq C \frac{(xy)^{\tlambda}}{(t^2+|x-y|^2)^{\tlambda + (\beta+1)/2}} \\
                &\leq  C \frac{(xy)^\tlambda}{(t+x)^\tlambda(t +|x-y|)^{\tlambda + \beta+1}}
                \leq   C \frac{y^\tlambda}{(t +|x-y|)^{\tlambda + \beta+1}}.
        \end{aligned}
\end{equation}
        Hence, \eqref{F1} is established.

        Observe that \eqref{F1'} follows in the same way as \eqref{F1} by keeping $\lambda$
        in the estimate \eqref{asterisco} instead of $\tlambda$.
        Also, \eqref{F1''} is obtained by using that $(xy)^{\tlambda}\leq C(x^2+y^2)^{\tlambda}$
        in \eqref{first} and without adding the factor $(xy/u)^{\tlambda/2}$ or $(xy/u)^\tlambda$
        in \eqref{asterisco2} and \eqref{asterisco3}.  Now notice that,
by the symmetry of the kernel $P_t^\lambda(x,y)=P_t^\lambda(y,x)$,
we could replace $y^\lambda$ by $x^\lambda$ in the right hand side of
\eqref{F1'}. In particular, we can replace $y^\lambda$ by $(x\wedge y)^\lambda$.
This observation combined with \eqref{F1''} give \eqref{F0'}.  Estimate \eqref{F0''} follows in the same way as before, but starting from the estimate
        $$|P_t^\lambda(x,y)|
                \leq   C t\int_0^\infty \frac{e^{-t^2/(4u)}}{u^{(1+2)/2}} W_u^\lambda(x,y)\,du.$$
We can derive \eqref{F0} by using \eqref{F0''} and \eqref{tiza} below.

        By applying \eqref{I derivada} it is not difficult to check that
        $$\partial_y W_u^\lambda(x,y)=
        \left( \frac{2\lambda u}{y}-y \right) \frac{1}{2u} W_u^\lambda(x,y) + \frac{x}{2u}W_u^{\lambda+1}(x,y).$$
        As above we can differentiate under the integral sign to get that $\partial_y\partial_t^\beta P_t^\lambda(x,y)$
        is equal to
        \begin{equation}\label{jajaja}
        \frac{1}{y} \partial_t^\beta P_t^\lambda(x,y)
                 +  \int_0^\infty \frac{\partial_t^\beta\left[ te^{-t^2/(4u)} \right]}{u^{5/2}}
                        \left[xW_u^{\lambda+1}(x,y) - yW_u^{\lambda}(x,y) \right]du.
        \end{equation}
        The first term above is controlled by \eqref{F1}, which gives the first term
        in estimate \eqref{F2}.
        The remaining integral in \eqref{jajaja} is splitted in
        two parts: $I+II=\int_0^{xy/2}+\int_{xy/2}^\infty$. By \eqref{beta lema} and \eqref{W origen},
        \begin{align*}
            |II| &\leq C \int_{xy/2}^\infty \frac{(x+y)e^{-(t^2+x^2+y^2)/(8u)}}{u^{(\beta+5)/2}} \left( \frac{xy}{u} \right)^\tlambda du
                \leq C \frac{(x+y)(xy)^\tlambda}{(t^2+x^2+y^2)^{\tlambda + (\beta+3)/2}} \\
               & \leq  C \frac{y^{\tlambda-1}}{(t+|x-y|)^{\tlambda + \beta+1}}.
        \end{align*}
        On the other hand, by taking into account \eqref{I infinito}, we deduce that, whenever $\frac{xy}{2u} \geq 1$,
        \begin{align*}
            \left| xW_u^{\lambda+1}(x,y) - yW_u^{\lambda}(x,y) \right|
            &  = \frac{e^{-|x-y|^2/4u}}{\sqrt{4\pi u}}
                    \left| x \left[ 1 + \Psi_{\lambda+1/2}\left( \frac{xy}{2u}\right) \right] - y \left[ 1 +
                    \Psi_{\lambda-1/2}\left( \frac{xy}{2u}\right) \right] \right| \\
            &  \leq C \frac{e^{-|x-y|^2/(4u)}}{\sqrt{u}}
                    \left( |x-y|+\frac{u}{y}+\frac{u}{x}  \right).
        \end{align*}
        Then, by proceeding as above,
        \begin{align*}
            |I|&\leq C \left( \int_0^{xy/2} \frac{e^{-(t^2+|x-y|^2)/(8u)}}{u^{(\beta+4)/2}}\,du
                              + \left(\frac{1}{y}+\frac{1}{x}\right)\int_0^{xy/2} \frac{e^{-(t^2+|x-y|^2)/(8u)}}{u^{(\beta+3)/2}}\,du \right) \\
               & \leq  C \left( \frac{1}{(t+|x-y|)^{\beta+2}} + \frac{y^{\tlambda-1}}{(t+|x-y|)^{\tlambda + \beta+1}} +
                 \frac{1}{x} \int_0^{xy/2} \frac{e^{-(t^2+|x-y|^2)/(8u)}}{u^{(\beta+3)/2}}\,du\right).
        \end{align*}
        To finish the proof of  \eqref{F2} we just note that the
        last integral is bounded by
        $$ \frac{C}{y} \int_0^{xy/2} \frac{e^{-(t^2+|x-y|^2)/(8u)}}{u^{(\beta+3)/2}}\,du
            \leq C \frac{y^{\tlambda-1}}{(t+|x-y|)^{\tlambda + \beta+1}}, \quad\hbox{when}~x \geq y/2, $$
        and by
        \begin{align*}
            \frac{C}{x} \int_0^{xy/2} \frac{e^{-(t^2+|x-y|^2)/(8u)}}{u^{(\beta+3)/2}} \left(\frac{xy}{u}\right)du
               & \leq   Cy \int_0^{xy/2} \frac{e^{-(t^2+y^2)/(32u)}}{u^{(\beta+5)/2}} \,du \\
                &\leq \frac{C}{(t^2+y^2)^{(\beta+2)/2}}
                \leq \frac{C}{(t+|x-y|)^{\beta+2}}, \quad \hbox{when}~x < y/2.
        \end{align*}
        \end{proof}

    \begin{Cor}\label{Lem permuta}
        Let $\lambda,\beta>0$ and let $f\in L_{(\beta+\lambda)/2}$. Then,
        for all $x \in \R_+$ and $t>0$,
        $$\partial_t^\beta P_t^\lambda f(x)
            = \int_0^\infty \partial_t^\beta P_t^\lambda(x,y)f(y)\,dy.$$
    \end{Cor}

        \begin{proof}
        According to \eqref{F0'} in Lemma~\ref{Lem fractPoiss} it follows that
        we can write
        $$\partial_t^m P_t^\lambda f(x)
            = \int_0^\infty \partial_t^m P_t^\lambda(x,y)f(y)\,dy.$$
        We conclude by applying the definition of $\partial_t^\beta$, see \eqref{fractional derivative}.
    \end{proof}

\begin{proof}[Proof of Proposition \ref{thm:harmonic extension}]
It is customary to verify that the Poisson kernel $P_t^\lambda(x,y)$
verifies the equation $\partial_{tt}P_t^\lambda(x,y)-(\Delta_\lambda)_xP_t^\lambda(x,y)=0$,
for all $x,y\in\R_+$ and $t>0$.
It remains to check that the derivatives in $t$ and in $x$ can enter
inside the integral \eqref{definicion de v} that defines $v$. Observe first
that, by Corollary \ref{Lem permuta} and the equation,
$$\partial_{t}^2 v(x,t)
    = \int_0^\infty \partial_{t}^2  P_t^\lambda(x,y)f(y)\,dy
    = -\int_0^\infty \partial_{x}^2  P_t^\lambda(x,y)f(y)\,dy
        +\frac{\lambda(\lambda-1)}{x^2}v(x,t).$$
Notice that the integrals above are both absolutely convergent and
$$\int_0^\infty \partial_{x}^2  P_t^\lambda(x,y)f(y)\,dy
    =\partial_{x} \int_0^\infty \partial_{x}  P_t^\lambda(x,y)f(y)\,dy.$$
By using the symmetry of the kernel it is immediate that we can
estimate $\partial_xP_t^\lambda(x,y)$ by the right hand side of \eqref{F2}
with $\beta=0$. Hence the last integral is absolutely convergent and
it is equal to $\partial_{x}^2  v(x,t)$. Thus $v$ solves
the equation.

Now we have to check the convergence to
    the boundary data. Let us write, for $x\in\R_+$,
    $$P_t^\lambda f(x)
        = \int_0^\infty [P_t^\lambda(x,y)-\mathbb{P}_t(x-y)]f(y)\,dy
            + \int_\R \mathbb{P}_t(x-y)f_0(y)\,dy,$$
    where $f_0(y)=f(y)$ for $y\geq0$, and $f_0(y)=0$ for $y < 0$, and
    $$\mathbb{P}_t(x)
        = \frac{1}{\pi}\frac{t}{t^2+x^2}
            =\frac{t}{2\sqrt{\pi}}\int_0^\infty\frac{e^{-t^2/(4u)}}{u^{3/2}}\W_u(x)\,du, \quad x\in \mathbb{R}, \ t>0,$$
    is the classical Poisson kernel on $\R$.
    Since $f_0$ is in $C^\alpha(\R)$,  by standard computations,
    $$\lim_{t \to 0^+} \int_\R \mathbb{P}_t(x-y)f_0(y)\,dy
        = f(x), \quad\hbox{for each}~x \in \R_+.$$
    By using the subordination formula \eqref{subord kernel}, for $x,y\in\R_+$ and $t>0$, we have
    $$P_t^\lambda(x,y)-\mathbb{P}_t(x-y)
        = \frac{t}{2\sqrt{\pi}} \int_0^\infty \frac{e^{-t^2/(4u)}}{u^{3/2}} [W_u^\lambda(x,y)-\W_u(x-y)] \,du.$$
    According to \eqref{25.1},
    \begin{align*}
        t\int_0^{xy} \frac{e^{-t^2/(4u)}}{u^{3/2}} \left|W_{u}^\lambda(x,y)-\W_{u}(x-y)\right| du
           & \leq C \frac{t}{xy} \int_0^{xy} \frac{e^{-t^2/(4u)}}{u} \,du\\
&\leq C\frac{t^{1/2}}{xy}\int_0^{xy}e^{-ct^2/u}\frac{du}{u^{3/4}}       \leq C \frac{t^{1/2}}{(xy)^{3/4}},
    \end{align*}
    and similarly, by \eqref{control clasico},
    \begin{align*}
        t \int_{xy}^\infty \frac{e^{-t^2/(4u)}}{u^{3/2}} \left|W_{u}^\lambda(x,y)-\W_{u}(x-y)\right| du
            \leq C t \int_{xy}^\infty \frac{e^{-t^2/(4u)}}{u^{2}} \,du
            \leq C \frac{t^{1/2}}{(xy)^{3/4}}.
    \end{align*}
    Also, by \eqref{control clasico} and \eqref{subord kernel},
\begin{equation}\label{tiza}
P_t^\lambda(x,y) \leq C \mathbb{P}_t(x-y).
\end{equation}
    These estimates lead to
    \begin{align*}
        \left| \int_0^\infty [P_t^\lambda(x,y)-\mathbb{P}_t(x-y)]f(y)\,dy \right|
            &\leq  C \left( \int_{(0,x/2)\cup (2x,\infty)} \frac{ty^\alpha}{t^2+(x-y)^2}  \,dy
                             + \int_{x/2}^{2x} \frac{t^{1/2} y^{\alpha}}{(xy)^{3/4}}  \,dy   \right) \\
       & \leq  C \left( t x^{\alpha-1} + t^{1/2} x^{\alpha-1/2} \right).
    \end{align*}
    Therefore the conclusion follows from the fact that
    $$\lim_{t \to 0^+} \int_0^\infty \left[P_t^\lambda(x,y)-\mathbb{P}_t(x-y)\right]f(y)\,dy
        = 0, \quad\hbox{for each}~ x \in \R_+.$$
        \end{proof}
        
        {\section{H\"older spaces $C^\alpha_+$}\label{sec:Holder}}
    \subsection{{Duality between Hardy spaces $H^p(\R_+)$ and H\"older spaces $C^\alpha_+$}} \label{sec:Hardy}

    Let $1/2<p \leq 1$. We define in this section the atomic Hardy space $H^p(\R_+)$
    associated to the Bessel operator $\Delta_\lambda$ on $\R_+$.
    A measurable function $a$ on $\R_+$ is called a $(p,\infty)$-atom when one of the following properties is satisfied
    \begin{itemize}
        \item[$(i)_\infty$] there exist $0 \leq b < c<\infty$ such that $\supp a \subset [b,c]$, $\displaystyle \int_0^\infty a(x)\,dx=0$ and $\|a\|_\infty \leq (c-b)^{-1/p}$;
        \item[$(ii)_\infty$] $a=b^{-1/p} \chi_{(0,b)}$, for some $b \in \R_+$.
    \end{itemize}
    A distribution $f \in\mathcal{S}'(\R_+)$, being $\mathcal{S}(\R_+)$ the Schwartz space,
     is said to be in $H^p_\infty(\R_+)$ provided that $f=\sum_{i=1}^\infty \alpha_i a_i$,
    in the sense of convergence in $\mathcal{S}'(\R_+)$, where, for every $i \in \N$, $a_i$ is a $(p,\infty)$-atom and $\alpha_i \in \C$,
    being $\sum_{i=1}^\infty |\alpha_i|^p<\infty$. The norm in $H^p_\infty(\R_+)$ is defined by
    $$\|f\|_{H^p_\infty(\R_+)} = \inf \left( \sum_{i=1}^\infty |\alpha_i|^p \right)^{1/p},$$
    where the infimum is taken over all possible decompositions of $f$ as above.

    The space $H^1_\infty(\R_+)$ was first considered by Fridli \cite{Fr} to study the local Hilbert transform.

    Note that if condition $(ii)_\infty$ above is replaced by the following one:
    \begin{itemize}
        \item[$(ii')_{\infty}$] there exists $b>0$ such that $\supp a \subset [0,b]$ and $\|a\|_\infty \leq b^{-1/p}$,
    \end{itemize}
    then the atomic Hardy space $H^p_\infty(\R_+)$ does not change. Indeed, suppose that $a$ is measurable function on $\R_+$
    such that $\supp a \subset [0,b]$ and $\|a\|_\infty \leq b^{-1/p}$, for some $b>0$. Then we can write
    $$a=2A_1 + b^{-1+1/p} \int_0^b a(y)\,dy\cdot A_2,$$
    where
    $$A_1=\frac{1}{2}\left(a-\frac{1}{b} \int_0^b a(y)\,dy \right) \chi_{(0,b)},\quad \text{and} \quad A_2=b^{-1/p}\chi_{(0,b)}.$$
    Note that $A_1$ and $A_2$ are $(p,\infty)$-atoms and $\displaystyle \left| b^{-1+1/p} \int_0^b a(y)\,dy \right| \leq 1$.

    Let $1<q<\infty$. We say that a measurable function $a$ on $\R_+$ is a $(p,q)$-atom when one of the following two conditions is satisfied:
    \begin{itemize}
        \item[$(i)_q$] there are $0 \leq b < c<\infty$ such that $\supp a\subset [b,c]$, $\displaystyle \int_0^\infty a(x)\,dx=0$ and $\|a\|_q \leq (c-b)^{1/q-1/p}$;
        \item[$(ii)_q$] $a=b^{-1/p} \chi_{(0,b)}$, for some $b \in \R_+$.
    \end{itemize}
    The atomic Hardy space $H^p_q(\R_+)$ is defined by using $(p,q)$-atoms as usual. As above, if the condition $(ii)_q$ is replaced
    by the corresponding property $(ii')_q$, the space $H^p_q(\R_+)$ does not change.

    We are going to show that $H^p_{q_1}(\R_+)=H^p_{q_2}(\R_+)$ algebraic and topologically, for every $1<q_1,q_2 \leq \infty$. In the sequel
    $H^p(\R)$ denotes the classical Hardy space on $\R$, see \cite{Duo, Ste2}. If $f \in \mathcal{S}'(\R_+)$ or $f$ is a function defined in $[0,\infty)$ we represent by
    $f_o$ the odd extension of $f$ to $\R$, properly understood in each case.

    \begin{Prop}\label{Prop oddHardy}
        Let $1/2<p \leq 1$, $1<q \leq \infty$, and $f \in \mathcal{S}'(\R_+)$.
        Then, $f \in H^p_q(\R_+)$ if and only if $f_o \in H^p(\R)$ and $\|f\|_{H^p_q(\R_+)} \sim \|f_o\|_{H^p(\R)}$.
    \end{Prop}

    \begin{proof}
        Suppose that $f \in H^p_q(\R_+)$. We write $f=\sum_{i=1}^\infty \alpha_i a_i$ in $\mathcal{S}'(\R_+)$ where $a_i$
        is a $(p,q)$-atom and $\alpha_i \in \C$, for every $i \in \N$, being $\sum_{i=1}^\infty |\alpha_i|^p<\infty$. Then,
        $f_o=\sum_{i=1}^\infty \alpha_i \tilde{a}_i$, where $\tilde{a}_i$ is the odd extension of $a_i$ to $\R$. It is clear that
        if $a$ is a $(p,q)$-atom such that $\supp a \subset \R_+$, then $\tilde{a}=\tilde{a}\chi_{(0,\infty)} + \tilde{a}\chi_{(-\infty,0)}$,
        and $\tilde{a}\chi_{(0,\infty)}$, $\tilde{a}\chi_{(-\infty,0)}$ are $(p,q)$-atoms for $H^p(\R)$. If $b>0$ and $a=b^{-1/p}\chi_{(0,b)}$,
        then $\tilde{a}=b^{-1/p}(\chi_{(0,b)}-\chi_{(-b,0)})$ and $\tilde{a}/2 $ is a $(p,q)$-atom for $H^p(\R)$. Hence, $f_o \in H^p(\R)$ and
        $\|f_o\|_{H^p(\R)} \leq C \|f\|_{H^p_q(\R_+)}$.

        Assume now that $f_o \in H^p(\R)$ is odd an $f_o=\sum_{i=1}^\infty \alpha_i a_i$ in $\mathcal{S}'(\R)$, where $a_i$
        is a $(p,q)$-atom for $H^p(\R)$ and $\alpha_i \in \C$, for every $i \in \N$, being $\sum_{i=1}^\infty |\alpha_i|^p<\infty$.
        Then, $f_o=\sum_{i=1}^\infty \alpha_i A_i$ in $\mathcal{S}'(\R)$, where $A_i(x)=(a_i(x)-a_i(-x))/2$, $x \in \R$, $i \in \N$, and
        $f=\sum_{i=1}^\infty \alpha_i \A_i$ in $\mathcal{S}'(\R_+)$, where $\A_i=A_i \chi_{(0,\infty)}$, $i \in \N$. If $a$ is a
        $(p,\infty)$-atom for $H^p(\R)$ such that
        \begin{itemize}
            \item $\supp a \subset [c,d] \subset [0,\infty)$, then $A(x)=(a(x)-a(-x))/2=a(x)/2$, $x \in [0,\infty)$,
            $\supp \A \subset [c,d]$ and $\|\A\|_q \leq (d-c)^{1/q-1/p}/2$, where $\A=A \chi_{(0,\infty)}$.
            \item $\supp a \subset [c,d] \subset (-\infty,0]$, then $A(x)=(a(x)-a(-x))/2=a(-x)/2$, $x \in [0,\infty)$,
            $\supp \A \subset [-d,-c]$ and $\|\A\|_q \leq  (d-c)^{1/q-1/p}/2$, where $\A=A \chi_{(0,\infty)}$.
            \item $\supp a \subset [c,d]$, $c<0<d$, then $A(x)=(a(x)-a(-x))\chi_{(0,\infty)}/2$, $x \in [0,\infty)$,
            is supported in $[0,\max\{-c,d\}]$ and $\|\A\|_q \leq \|a\|_q \leq  (d-c)^{1/q-1/p} \leq \max\{-c,d\}^{1/q-1/p}$.
        \end{itemize}
        Thus, $f \in H^p_q(\R_+)$ and $\|f\|_{H^p_q(\R_+)} \leq C \|f_o\|_{H^p(\R)}$.
    \end{proof}

    An immediate consequence of Proposition~\ref{Prop oddHardy} (and a well known result about
    atomic characterization of classical Hardy spaces, see \cite[Chapter~III,~5.6,~p.~130]{Ste2}) is the following

    \begin{Cor}\label{Cor eqHardy}
        Let $1/2<p \leq 1$ and $1<q \leq \infty$. Then, $H^p_\infty(\R_+)=H^p_q(\R_+)$ algebraic and topologically.
    \end{Cor}

    In the sequel we simply write $H^p(\R_+)$ to denote the Hardy space $H^p_\infty(\R_+)$.

    \begin{Prop}\label{Prop L2Hp}
        Let $1/2 < p \leq 1$. There exists $C>0$ such that, for every $f \in L^2(\R_+) \cap H^p(\R_+)$, we can find
        a sequence $\{\alpha_j\}_{j=1}^\infty$ of complex numbers and a sequence $\{a_j\}_{j=1}^\infty$ of $(p,2)$-atoms such that
        $\sum_{j=1}^\infty |\alpha_j|^p \leq C \|f\|_{H^p(\R_+)}^p$ and
        $f = \sum_{j=1}^\infty \alpha_j a_j$ in $L^2(\R_+)$.
    \end{Prop}

    \begin{proof}
        Let $f \in L^2(\R_+) \cap H^p(\R_+)$. Then, by Proposition~\ref{Prop oddHardy}, $f_o \in L^2(\R) \cap H^p(\R)$
        and $\|f\|_{H^p(\R_+)} \sim \|f_o\|_{H^p(\R)}$. According to \cite[Theorem 1.2]{ZH} there exist a sequence
        $\{\alpha_j\}_{j=1}^\infty$ of complex numbers and a sequence of $(p,2)$-atoms for $H^p(\R)$ such that
        $\sum_{j=1}^\infty |\alpha_j|^p \leq C \|f_o\|_{H^p(\R)}^p$ and
        $f_o = \sum_{j=1}^\infty \alpha_j a_j$ in $L^2(\R)$, where $C>0$ does not depend on $f$.
        As it was shown in the proof of Proposition~\ref{Prop oddHardy}, $f = \sum_{j=1}^\infty \alpha_j \mathbb{A}_j$ in $\mathcal{S}'(\R_+)$,
        where, for every $j \in \mathbb{N}\setminus \{0\}$,
        $\mathbb{A}_j(x) = \chi_{[0,\infty)}(x)(a_i(x)-a_i(-x))/2 $, $x \in [0,\infty)$, is a $(p,2)$-atom for $H^p(\R_+)$.
        Moreover,
        $\sum_{j=1}^\infty |\alpha_j|^p \leq C \|f\|_{H^p(\R_+)}^p$ and
        $f = \sum_{j=1}^\infty \alpha_j \mathbb{A}_j$ in $L^2(\R_+)$.
    \end{proof}

    We denote by $C^\alpha(\R)$ the classical space of $\alpha$-H\"older continuous functions on $\R$, $0<\alpha<1$,
    and by $\mathrm{Lip}(\R)$ the space of Lipschitz functions on $\R$.

    \begin{Lem}\label{Lem oddHolder}
         Let $\alpha \in (0,1)$. A function $f$ is in $C^\alpha_+$ (resp. in $\mathrm{Lip}_+$)  if and only if $f_o$
         is in $C^\alpha(\R)$ (resp. in $\mathrm{Lip}(\R)$).
    \end{Lem}

    \begin{proof}
        Let $f \in C^\alpha_+$. It is clear that $|f_o(x)-f_o(y)| \leq C |x-y|^\alpha$, $x,y \in \R$ and $xy \geq 0$.
        If $x<0<y$, we can write
        \begin{align*}
            |f_o(x)-f_o(y)|
                \leq & |-f(-x)-f(y)|
                \leq |f(-x)-f(y)| +2|f(-x)|
                \leq C (|x+y|^\alpha+|x|^\alpha)
                \leq C |x-y|^\alpha.
        \end{align*}
        Hence, $f_o \in C^\alpha(\R)$. Moreover, if $f_o \in C^\alpha(\R)$ it is obvious that
        $f_o \chi_{(0,\infty)} \in C^\alpha_+$. When $f \in \mathrm{Lip}_+$ we can proceed similarly.
    \end{proof}

    By proceeding as in the proof of \cite[Theorem 5.1, p. 213]{Tor} and using the John--Nirenberg Lemma we can see that $f \in C^\alpha_+$
    if and only if, for every $1 \leq p < \infty$, there exists $C_p>0$ such that
    \begin{equation*}\label{N1}
        \frac{1}{|I|^{1+\alpha}} \int_I |f(y)-f_I|^p \,dy \leq C_p,
    \end{equation*}
    for every bounded interval $I \subset \R_+$, and
    \begin{equation*}\label{N2}
        \frac{1}{|I|^{1+\alpha}} \int_I |f(y)|^p \,dy \leq C_p,
    \end{equation*}
    for every $I=(0,b)$, $b>0$. Moreover,
    $\inf\{C_p : \text{ the above conditions hold}\}
    \sim\|f\|_{C^\alpha_+}$.

    By following exactly the same steps as in the classical case (see for example \cite[p.~130]{Ste2}) it is possible to
    show the following duality result. We omit the details.

    \begin{Th}\label{Prop duality}
        Let $1/2<p < 1$. The dual space of the Hardy space $H^p(\R_+)$ coincides with
        the space $C^{(1-p)/p}_+$ in the following sense. For every $T \in \left(H^p(\R_+)\right)'$
        there exists $f \in C^{(1-p)/p}_+$ such that
        $$\langle T,g \rangle = \int_0^\infty f(y)g(y)\,dy, \quad g \in \spann\{(p,\infty)\text{-atoms}\},$$
        and, if $f \in C^{(1-p)/p}_+$ then the mapping $T_f$ defined by
        $$\langle T_f,g \rangle = \int_0^\infty f(y)g(y)\,dy, \quad g \in \spann\{(p,\infty)\text{-atoms}\},$$
        can be extended to $H^p(\R_+)$ as an element of $\left(H^p(\R_+)\right)'$.
    \end{Th}

\begin{Rem}
In the case $p=1$ the result of
Theorem \ref{Prop duality} is already known, see \cite{BCFR2}.
\end{Rem}

\quad

    \subsection{{Characterizations of the H\"older space $C_+^\alpha$}} \label{sec:core}

    {In this section we characterize the H\"older spaces $C_+^\alpha$ by using the usual
    pointwise condition, the growth property \eqref{cara},
     Campanato-type conditions and also Carleson measures involving derivatives
    of the Bessel--Poisson semigroup.}
    
{    \begin{Th}[Characterizations of the H\"older space
$C^\alpha_+$]\label{Th1.1}
        Let $\lambda,\beta>0$ and $0<\alpha<1$ be such that $\alpha<\lambda \wedge \beta=\min\{\lambda,\beta\}$.
        Assume that $f$ is a continuous function in $L_{\lambda\wedge\beta+1}$.
        The following assertions are equivalent.
        \begin{itemize}
            \item[$(i)$] (Pointwise)  $f \in C^\alpha_+$.
            \item[$(ii)$] (Campanato-type) For every $1 \leq p < \infty$, there exists $C_p>0$ such that
    \begin{equation}\tag{M1}\label{M1}
        \frac{1}{|I|^{1+\alpha}} \int_I |f(y)-f_I|^p dy \leq C_p,
    \end{equation}
    for every bounded interval $I \subset \R_+$, and
    \begin{equation}\tag{M2}\label{M2}
        \frac{1}{|I|^{1+\alpha}} \int_I |f(y)|^p dy \leq C_p.
    \end{equation}
    for every $I=(0,b)$, $b>0$. Here $\displaystyle f_I=\frac{1}{|I|}\int_I f(y)\,dy$, for every bounded interval $I$.
            \item[$(iii)$] (Fractional derivatives of Poisson semigroup)
            $\|t^{\beta}\partial_t^\beta P_t^\lambda f\|_{L^\infty(\R_+)}\leq Ct^\alpha$, for all $t>0$.
            \item[$(iv)$] (Fractional Carleson measure) We have
            \begin{equation}\label{alphaCarl}
                [d\mu_f]_\alpha:=\sup_{I\subset \R_+}\int_0^{|I|}\int_I\big|t^\beta
                \partial_t^\beta P_t^\lambda f(x) \big|^2 \frac{dx\,dt}{t}<\infty,
            \end{equation}
            where the supremum is taken over all the bounded intervals in $\R_+$.
        \end{itemize}
        Moreover, the following quantities:
        $\|f\|_{C^\alpha_+}$, $\inf\{C_p:\eqref{M1}\,\hbox{and}\,\eqref{M2}\,\hbox{hold}\}$,
        $\|t^{\beta-\alpha} \partial_t^\beta P_t^\lambda f\|_{ L^\infty(\R_+ \times (0,\infty))}$
        and $[d\mu_f]_\alpha^{1/2}$ are equivalent.
    \end{Th}}

 {
    \begin{Rem}
        The Campanato-type characterization of $C^\alpha_+$ gives at $\alpha=0$ the $BMO$ space associated
        to $\Delta_\lambda$, which was already studied in \cite{BCFR2,BCFR1}. It can be proved that
        such a $BMO$ space is caracterized by condition $(iv)$
        of Theorem \ref{Th1.1} with $\alpha=0$, for any $\beta>0$;
        and that condition $(ii)$ implies $(iii)$ with $\alpha=0$. Analogous questions can be posed when $\alpha=1$,
        that is, for the space $\mathrm{Lip}_+$.
    \end{Rem}
    \begin{Rem}
        It is useful in applications to have characterizations
        of the vector-valued space $C^\alpha_{+,\mathbb{B}}$, which is the space of
        $C^\alpha_+$-functions taking values in a Banach space $\mathbb{B}$. One could define vector-valued
        versions of the fractional area and square functions appearing below and try to characterize,
        via geometric conditions,
        the Banach spaces $\mathbb{B}$ for which such operators are bounded in $C^\alpha_{+,\mathbb{B}}$. Some of these
        questions will be addressed in a forthcoming work.
    \end{Rem}}

{The proof of Theorem \ref{Th1.1} is as follows:  $(i)\Longleftrightarrow(ii)$ 
was already done in Subsection \ref{sec:Hardy}; for $(i)\implies(iii)$
    we use estimates and properties of the Poisson kernel
    given in Section \ref{sec:Poisson}; $(iii)\implies(iv)$ is trivial; the deep part is $(iv)\implies(ii)$.
    For this we need the auxiliary atomic Hardy space considered 
    in Subsection \ref{sec:Hardy} (whose dual is $C^\alpha_+$) and a reproducing formula
    involving \eqref{alphaCarl} in order to conclude. The latter scheme of proof
    is classical (see \cite[Chapter~IV]{Ste2}),
    though here becomes more technical because we are using fractional derivatives. Even if for our application (Theorem A)
    we just need the statement $(iii)$, we need to go through $(iv)$ in order to close the argument. In the classical
    case of the H\"older space on $\R^N$ one can prove directly that $(iii)\implies(i)$
    because harmonic functions on the upper half space satisfy the simpler equation $v_{tt}+\Delta_xv=0$, see \cite[Chapter~V]{Stein}.
    Moreover, in contrast with the classical situation of the Laplacian, the constant functions are not invariant
    for the semigroups of operators $W_t^\lambda$ and $P_t^\lambda$ due to the presence
    of the potential $\frac{\lambda(\lambda-1)}{x^2}$.
    This fact makes the proofs of our results more involved.}

   \begin{proof}[{Proof of Theorem~\ref{Th1.1}, $(i) \Longrightarrow (iii)$}]
   Assume that $f \in C^\alpha_+$. According to Corollary~\ref{Lem permuta} and \eqref{F1''},
    \begin{align*}
        | t^\beta \partial_t^\beta P_t^\lambda f(x) |
            &\leq  \int_0^\infty |t^\beta \partial_t^\beta P_t^\lambda(x,y)| |f(y)-f(x)|\, dy
                    + |f(x)| \left| \int_0^\infty t^\beta \partial_t^\beta P_t^\lambda(x,y) \,dy \right| \\
            &\leq  \|f\|_{C^\alpha_+} \left( \int_0^\infty\frac{t^\beta}{(t+|x-y|)^{\beta+1}} |x-y|^\alpha\, dy
                    + x^\alpha \left| \int_0^\infty t^\beta \partial_t^\beta P_t^\lambda(x,y) \,dy \right| \right).
    \end{align*}
Certainly the first term above is bounded by $Ct^\alpha$. The second one is
handled by applying the following result, whose proof will be given in
Section \ref{sec:demostraciones}, with $\delta=\alpha$.

\begin{Lem}\label{Lem intPoiss}
Let $0 \leq \delta \leq 2$, $\delta < \beta$ and $\lambda>0$. Then, for all  $x \in \R_+$ and $t>0$,
$$\left| \int_0^\infty t^\beta\partial_t^\beta P_t^\lambda(x,y) \,dy \right|
\leq C \left( \frac{t}{x}\right)^\delta.$$
\end{Lem}

Thus, $(iii)$ is established.
\end{proof}

\begin{proof}[{Proof of Theorem~\ref{Th1.1},  $(iv) \Longrightarrow (i)$}]
As we said  {before}, this is the most technical part of the
proof. The rest of this section is devoted to it. In order to make
the presentation more readable,
    we will omit a couple of proofs that will be given later in Section \ref{sec:demostraciones}.
    Assume that \eqref{alphaCarl} holds. According to
    Theorem~\ref{Prop duality}, $f \in C^\alpha_+$ provided that the mapping $T_f$ defined there
    can be extended to $H^p(\R_+)$ as a bounded operator from $H^p(\R_+)$ into $\C$, being $p=1/(1+\alpha)$.
    This can be established by using Propositions~\ref{Prop1}, \ref{Prop2} and \ref{Prop3} below. Let us explain how to do it.

    Consider the fractional area function $S_{\lambda}^\beta$ given by
   $$S_{\lambda}^\beta(f)(x)
        = \left( \int_{\G_+(x)} \left| t^\beta \partial_t^\beta P_t^\lambda f(y) \right|^2 \,\frac{dt\,dy}{t^2} \right)^{1/2},
        \quad x \in \R_+,$$
   where $\G_+(x)$ is the \textit{positive} cone
   $\G_+(x)=\{(y,t) \in \R_+ \times (0,\infty) : |x-y|<t\}$.
   We claim that $S_{\lambda}^\beta$ is bounded from $L^2(\R_+)$ into itself. Indeed, let
   $$g_\lambda^\beta(f)(x)
        = \left( \int_0^\infty \big| t^\beta \partial_t^\beta P_t^\lambda f(x) \big|^2 \,\frac{dt}{t} \right)^{1/2}, \quad x \in \R_+,$$
       the $\beta$-Littlewood-Paley function  associated with $\{P_t^\lambda\}_{t>0}$.
   We have that
   \begin{align*}
        \|S_{\lambda}^\beta(f)\|_{L^2(\R_+)}^2
            &=  \int_0^\infty \int_0^\infty \int_{|x-y|<t} \left| t^\beta \partial_t^\beta P_t^\lambda f(y) \right|^2 \,\frac{dt\,dy\,dx}{t^2}\\
           & \leq  C \int_0^\infty \int_0^\infty \left| t^\beta \partial_t^\beta P_t^\lambda f(y) \right|^2\, \frac{dt\,dy}{t}
            = C \|g_{\lambda}^\beta(f)\|_{L^2(\R_+)}^2,
   \end{align*}
   for $ f \in L^2(\R_+)$. To conclude we need the following result, whose proof is given
   in Section~\ref{sec:demostraciones}.

   \begin{Lem}\label{Lem hankel}
        Let $\lambda, \beta>0$ and $f \in L^2(\R_+)$. Then,
        $$\partial_t^\beta P_t^\lambda f
                = h_\lambda \left( e^{-i \pi \beta} y^\beta e^{-yt} h_\lambda f (y) \right),\quad\hbox{for}~t>0.$$
    \end{Lem}

With this result and Plancherel equality for Hankel transforms (see \cite[(3)]{Ze}) we get
   \begin{align*}
        \|g_{\lambda}^\beta(f)\|_{L^2(\R_+)}^2
            &=  \int_0^\infty \int_0^\infty \left| h_\lambda \left( (ty)^\beta e^{-yt} h_\lambda(f) \right)(x) \right|^2\, \frac{dx\,dt}{t}\\
           & =  \int_0^\infty \int_0^\infty (ty)^{2\beta} e^{-2ty} \left|    h_\lambda(f) (y) \right|^2\, \frac{dt\,dy}{t}
            =  \frac{\G(2\beta)}{2^{2\beta-1}} \|f\|_{L^2(\R_+)}^2.
   \end{align*}
   Hence, $g_{\lambda}^\beta$, and also $S_{\lambda}^\beta$, are bounded from $L^2(\R_+)$ into itself.

The first step towards the proof of $(iv)\Longrightarrow(i)$ is
to prove that $S_{\lambda}^\beta$ defines a bounded operator from $H^p(\R_+)$ into $L^p(\R_+)$. In order to do this we
study the action of the area function just defined on the $(2,p)$-atoms
introduced in Section \ref{sec:Hardy}.

   \begin{Prop}\label{Prop1}
        Let $\beta$, $\lambda>0$ and $1/(\tlambda+1) < p \leq 1$. Then there exists $C>0$ such that for every
        $(2,p)$-atom $a$,
        \begin{equation}\label{objt1}
            \| S_{\lambda}^\beta(a) \|_{L^p(\R_+)} \leq C.
        \end{equation}
   \end{Prop}

    \begin{proof}
        Suppose firstly that $a$ is an $(2,p)$-atom that satisfies $(ii)_2$,
        namely, such that for a certain $b>0$, $a = b^{-1/p} \chi_{(0,b)}$. We can write
        \begin{align*}
            \int_0^\infty \left| S_{\lambda}^\beta(a)(x) \right|^p\, dx
                = & \int_0^{2b} \left| S_{\lambda}^\beta(a)(x) \right|^p\, dx
                + \int_{2b}^\infty \left| S_{\lambda}^\beta(a)(x) \right|^p\, dx
                =:  J_1 + J_2.
        \end{align*}
        Since $S_{\lambda}^\beta$ is bounded from $L^2(\R_+)$ into itself, H\"older's inequality leads to
        \begin{align*}
            J_1
                \leq & \left( \int_0^{2b} \left| S_{\lambda}^\beta(a)(x) \right|^2 dx \right)^{p/2} (2b)^{1-p/2}
                \leq C \left( \int_0^{b} \left| a(y) \right|^2 dx \right)^{p/2} b^{1-p/2}
                \leq C b^{p/2-1} b^{1-p/2}
                \leq C.
        \end{align*}
        To estimate $J_2$ we make the following observation.  According to \eqref{F1} and since $|x-z| \leq t + |y-z|$, when $x,y,z \in \R_+$, $t>0$ and
        $|x-y| \leq t$, we deduce that
        \begin{align*}
            \Big\| t^\beta \partial_t^\beta P_t^\lambda(y,z) &\Big\|^2_{L^2\left( \G_+(x) , \frac{dtdy}{t^2} \right)}
                \leq  C \int_0^\infty \int_{|x-y|<t} \frac{t^{2\beta-2}z^{2\tlambda}}{(t+|z-y|)^{2\tlambda + 2\beta+2}}\, dy \,dt \\
                &\leq  C z^{2\tlambda} \left( \int_0^{|x-z|} \int_{|x-y|<t}
                \frac{t^{2\beta-2}}{|x-z|^{2\tlambda + 2\beta+2}}\, dy\, dt
                + \int_{|x-z|}^\infty  \int_{|x-y|<t} \,\frac{dy \,dt}{t^{2\tlambda + 4}}  \right)\\
                &\leq  C \frac{z^{2\tlambda}}{|x-z|^{2\tlambda +2}}, \quad \hbox{for all}~x,z \in \R_+,~x \neq z.
        \end{align*}
        Hence, by Cauchy--Schwartz inequality and by taking into account that $p>1/(\tlambda+1)$,
        \begin{align*}
            J_2
                & \leq C \int_{2b}^\infty \left( \int_0^b a(z)\frac{z^{\tlambda}}{|x-z|^{\tlambda+1}}  dz \right)^p dx
                  \leq C \int_{2b}^\infty \|a\|_2^p \left( \int_0^b \frac{z^{2 \tlambda}}{|x-z|^{2 \tlambda+2}}  dz \right)^{p/2} dx \\
                & \leq C b^{p/2-1} \int_{2b}^\infty \frac{dx}{x^{(\tlambda+1)p}} \left( \int_0^b z^{2 \tlambda}dz \right)^{p/2}
                  \leq C b^{p/2-1} b^{-p(\tlambda+1)+1} b^{(2\tlambda +1)p/2}
                  \leq C.
        \end{align*}
        We conclude \eqref{objt1} for this type of atoms.

        Assume now that $a$ is a $(2,p)$-atom  that satisfies $(i)_2$, namely, such that for certain $0 \leq b < c < \infty$,
        $\supp a \subset I = (b,c)$, $\|a\|_2 \leq |I|^{1/2-1/p}$, where $|I|=c-b$, and
        $\int_0^\infty a(x)\,dx=0$. We denote by $2I=\R_+ \cap (x_I-|I|,x_I+|I|)$, where $x_I=(b+c)/2$
        is the center of $I$.
        We split the integral as
        \begin{align*}
            \int_0^\infty \left| S_{\lambda}^\beta(a)(x) \right|^p dx
                = & \int_{2I} \left| S_{\lambda}^\beta(a)(x) \right|^p dx
                + \int_{\R_+ \setminus 2I} \left| S_{\lambda}^\beta(a)(x) \right|^p dx
                =:   J_3 + J_4.
        \end{align*}
        By using the $L^2$-boundedness as above we can show that $J_3 \leq C$. On the other hand, we can write
        $$P_t^\lambda a (y)
            = \int_0^\infty \left[ P_t^\lambda(y,z) - P_t^\lambda(y,x_I) \right] a(z)\, dz, \quad  y \in \R_+,~t>0.$$
        Minkowski's inequality then leads to
        \begin{align}
            S_{\lambda}^\beta(a)(x)
               & =  \left( \int_{\G_+(x)} \left| \int_I \left[ t^\beta \partial_t^\beta P_t^\lambda(y,z) -
                    t^\beta \partial_t^\beta P_t^\lambda(y,x_I) \right] a(z)\,dz \right|^2\, \frac{dt\,dy}{t^2} \right)^{1/2}\nonumber \\
               & =  \left( \int_{\G_+(x)} \left| \int_I a(z) \int_{I_z}t^\beta \partial_t^\beta \partial_u
                    P_t^\lambda(y,u) \,du\, dz \right|^2 \,\frac{dt\,dy}{t^2} \right)^{1/2}\nonumber\\
                &\leq \int_I |a(z)| \int_{I_z} \left( \int_{\G_+(x)} \left|  t^\beta \partial_t^\beta \partial_u
                    P_t^\lambda(y,u)  \right|^2 \,\frac{dt\,dy}{t^2} \right)^{1/2}\,du\, dz \nonumber\\
                &\leq \frac{1}{|I|^{1/p}} \int_I \int_{I_z} \left( \int_{\G_+(x)} \left|  t^\beta \partial_t^\beta \partial_u
                    P_t^\lambda(y,u)  \right|^2\, \frac{dt\,dy}{t^2} \right)^{1/2}\,du \,dz,\label{pepino}
        \end{align}
        where, for every $z \in I$, $I_z=[z,x_I]$, when $z<x_I$, and $I_z=[x_I,z]$, when $z>x_I$.
        According to \eqref{F2} and as in the previous case we get
        \begin{align*}
            \int_{\G_+(x)} \left|  t^\beta \partial_t^\beta \partial_u P_t^\lambda(y,u)  \right|^2\, \frac{dt\,dy}{t^2}
            &\leq  C \Bigg( \int_{\G_+(x)} \frac{t^{2\beta-2}u^{2\tlambda-2}}{(t+|y-u|)^{2\tlambda+2\beta+2}}\, dt\,dy\\
            &\qquad\qquad  + \int_{\G_+(x)} \frac{t^{2\beta-2}}{(t+|y-u|)^{4+2\beta}} \,dt\, dy \Bigg) \\
            &\leq  C \left( \frac{u^{2\tilde{\lambda}-2}}{|x-u|^{2\tilde{\lambda}+2}}+ \frac{1}{|x-u|^4}\right).
        \end{align*}
        Notice that if $z \in I$, $u \in I_z$ and $x \in \R_+\setminus 2I$, then $|x-u| \sim |x-x_I|$.
        Then, by plugging the last estimate into \eqref{pepino},
        \begin{align*}
            S_\lambda^\beta(a)(x)
                &\leq C \left( \int_I |a(z)| \int_{I_z} \frac{u^{\tlambda-1}}{|x-u|^{\tlambda+1}} \,du\, dz
                            + \int_I|a(z)| \int_{I_z} \frac{1}{|x-u|^{2}} \,du\, dz\right) \\
                & \leq C \|a\|^2_2
                            \left[ \left( \int_I \left( \int_{I_z} \frac{u^{\tlambda-1}}{|x-u|^{\tlambda+1}} du \right)^2 dz \right)^{1/2}
                                +  \left( \int_I \left( \int_{I_z} \frac{1}{|x-u|^{2}} du \right)^2 dz \right)^{1/2} \right] \\
                & \leq C |I|^{1/2-1/p} \left[ \frac{1}{|x-x_I|^{\tlambda+1}} \left( \int_I |z^{\tlambda}-x_I^\tlambda|^2 dz \right)^{1/2}
                                                + \frac{1}{|x-x_I|^{2}} \left( \int_I |z-x_I|^2 dz \right)^{1/2}  \right],
        \end{align*}
        for $x \in \R_+ \setminus 2I$.
        Observe that in the first inequality above we have used that $(a+b)^\alpha \leq a^\alpha + b^\alpha$, when $a,b>0$ and $0 < \alpha \leq 1$.
It is clear that
        $$\left( \int_I |z-x_I|^2 \,dz \right)^{1/2}
            \leq |I|^{3/2}.$$
        On the other hand, since $\tlambda \in (0,1]$ we have that
        \begin{align*}
            \int_I |z^{\tlambda}-x_I^{\tlambda}|^2\, dz
                & = \int_{x_I}^{x_I + |I|/2} |z^{\tlambda}-x_I^{\tlambda}|^2 \,dz + \int_{x_I - |I|/2}^{x_I} |x_I^{\tlambda}-z^{\tlambda}|^2 \,dz \\
                & = \int_0^{|I|/2} |(x_I+z)^{\tlambda}-x_I^{\tlambda}|^2 \,dz + \int_0^{|I|/2} |x_I^{\tlambda}-(x_I-z)^{\tlambda}|^2 \,dz\\
                &                 \leq 2 \int_0^{|I|/2} z^{2\tlambda}\, dz
                  \leq C |I|^{2\tlambda+1}.
        \end{align*}
        Then, we get
        $$S_\lambda^\beta(a)(x)
            \leq C |I|^{1/2-1/p} \left( \frac{|I|^{\tlambda+1/2}}{|x-x_I|^{\tlambda+1}} + \frac{|I|^{3/2}}{|x-x_I|^{2}} \right)$$
        Consequently, we deduce that
        \begin{align*}
            J_4
                & \leq C |I|^{p/2-1} \int_{\R_+ \setminus 2I} \left( \frac{|I|^{\tlambda+1/2}}{|x-x_I|^{\tlambda+1}} + \frac{|I|^{3/2}}{|x-x_I|^{2}} \right)^{p} dx \\
                & \leq C |I|^{p/2-1} \int_{\R_+ \setminus 2I} \frac{|I|^{p(\tlambda+1/2)}}{|x-x_I|^{p(\tlambda+1)}} + \frac{|I|^{3/2p}}{|x-x_I|^{2p}}  dx \\
                & \leq C |I|^{p/2-1} \left( |I|^{p(\tlambda+1/2)-p(\tlambda+1)+1} + |I|^{3/2p-2p+1} \right)
                  \leq C.
        \end{align*}
        Thus, \eqref{objt1} is proved for atoms satisfying $(i)_2$.
    \end{proof}

    \begin{Cor}\label{Cor L2Hp}
        Let $\beta, \lambda >0$ and $1/(1+\tlambda) < p \leq 1$. Then, there exists $C>0$ such that for every
        $f \in L^2(\R_+) \cap H^p(\R_+)$,
        $$\|S_\lambda^\beta(f)\|_{L^p(\R_+)}
            \leq C \|f\|_{H^p(\R_+)}.$$
    \end{Cor}

    \begin{proof}
        Let $f \in L^2(\R_+) \cap H^p(\R_+)$. According to Proposition~\ref{Prop L2Hp} there exist a sequence
        $\{\alpha_j\}_{j=1}^\infty$ of complex numbers and a sequence $\{a_j\}_{j=1}^\infty$ of $(p,2)$-atoms
        such that
        $\sum_{j=1}^\infty |\alpha_j|^p \leq C \|f\|_{H^p(\R_+)}^p$ and
        $f = \sum_{j=1}^\infty \alpha_j a_j$ in $L^2(\R_+)$. Here $C>0$ does not depend on $f$.
        Since $S_\lambda^\beta$ is bounded from $L^2(\R_+)$ into $L^2(\R_+)$ we deduce that
        $$S_\lambda^\beta(f)
            \leq \sum_{j=1}^\infty |\alpha_j| S_\lambda^\beta(a_j).$$
        Then, Proposition~\ref{Prop1} leads to
        $$\|S_\lambda^\beta(f)\|_{L^p(\R_+)}
                \leq \sum_{j=1}^\infty |\alpha_j|^p \|S_\lambda^\beta(a_j)\|^p_{L^p(\R_+)}
                \leq C \sum_{j=1}^\infty |\alpha_j|^p.$$
        Hence, we get  $\|S_\lambda^\beta(f)\|_{L^p(\R_+)}
                \leq C \|f\|_{H^p(\R_+)}$.
    \end{proof}

    The second step is to recall the following result.

    \begin{Prop}\label{Prop2}
        Let $F$ and $G$ be measurable functions on $\R_+ \times (0,\infty)$. We define, for every $0<p \leq 1$,
        $$\Phi_p(F)
            = \sup_I \left( \frac{1}{|I|^{(2-p)/p}} \int_0^{|I|} \int_I \left| F(y,t) \right|^2 \frac{dy\,dt}{t}\right)^{1/2},$$
        where the supremum is taken over all the intervals $I \subset \R_+$, and
        $$\Psi_p(G)
            = \left( \int_0^\infty \left( \int_{\G_+(x)} \left| G(y,t) \right|^2 \frac{dy\,dt}{t^2} \right)^{p/2} dx \right)^{1/p}.$$
        Then, there exists a constant $C$ such that for all $F$ and $G$,
        $$\int_0^\infty \int_0^\infty \left| F(x,t) \right| \left| G(x,t) \right| \frac{dx\,dt}{t}
            \leq C \Phi_p(F) \Psi_p(G).$$
    \end{Prop}

    \begin{proof}
        This can be proved by using \cite[p. 279]{HSV}. To this end we need to replace the full cone $\G(x)$ by the positive cone $\G_+(x)$.
        The latter can be done  with the technique of the proof of \cite[Proposition~4.9]{BCFR1}. Details are omitted.
    \end{proof}

    The right choice of $F$ and $G$ in the Proposition above is dictated by the following polarization
    equality involving fractional derivatives of the Poisson semigroup.

    \begin{Prop}\label{Prop3}
        Let $\lambda, \beta>0$. Assume that for some $f\in L_{(\lambda\wedge \beta) / 2}$ the estimate \eqref{alphaCarl} holds.
        Let $a$ be a bounded function with compact support in $[0,\infty)$. Then,
        $$\int_0^\infty \int_0^\infty
                t^\beta \partial_t^\beta P_t^\lambda f (x)
       \overline{t^\beta \partial_t^\beta P_t^\lambda a (x)} \,\frac{dx\,dt}{t}
                = \frac{e^{2 \pi i \beta} \G(2\beta)}{2^{2\beta}} \int_0^\infty f(x) \overline{a(x)}\,dx.$$
    \end{Prop}

    The proof of Proposition \ref{Prop3} is quite technical and long, so it is presented in Section \ref{sec:demostraciones}. Now $(iv)\Longrightarrow(i)$ follows readily from Proposition \ref{Prop3}, Proposition
    \ref{Prop2} with $F(x,t)=t^\beta \partial_t^\beta P_t^\lambda f(x)$ and
    $G(x,t)=t^\beta \partial_t^\beta P_t^\lambda g(x)$,
    where $g \in \spann\{(p,\infty)-\text{atoms}\}$, and Corollary \ref{Cor L2Hp}.
\end{proof}
     {\section{The proof of Theorem A}\label{sec:Poisson and fractional}}

    The starting points for proving Theorem \ref{Th1.5} are formulas \eqref{spectral positive} and \eqref{spectral negative}.
    Both are valid for functions in the class $\mathcal{S}_\lambda$.
    If $u\in L_\sigma\cap C_+^\alpha$ and $0<2\sigma<\alpha<1$ then \eqref{spectral positive} holds and, moreover,
    it can be checked that for every $\lambda \geq 1$,
    $$\Delta_\lambda^\sigma u(x)=\int_0^\infty\big(u(x)-u(y))\widetilde{K}_\sigma^\lambda(x,y)\,dy+u(x)\widetilde{B}_\sigma^\lambda(x),$$
    being
    $$\widetilde{K}_\sigma^\lambda(x,y)= \frac{1}{-\G(-2\sigma)} \int_0^\infty P_t^\lambda(x,y) \frac{dt}{t^{1+2\sigma}},$$
    and
    $$\widetilde{B}_\sigma^\lambda(x)
            = \frac{1}{\G(-2\sigma)} \int_0^\infty \big(P_t^\lambda1(x) - 1 \big)\frac{dt}{t^{1+2\sigma}}.$$
    Indeed, by the subordination formula \eqref{subord kernel} and the properties of the Gamma function, it
    is easily seen that $\widetilde{K}_\sigma^\lambda(x,y) = K_\sigma^\lambda(x,y)$ and
    $\widetilde{B}_\sigma^\lambda(x) = B_\sigma^\lambda(x)$. Using Fubini's theorem
    and Corollary \ref{Thm:Calpha}
    we readily get the validity of \eqref{spectral positive}.

    \begin{proof}[Proof of Theorem \ref{Th1.5}]~

    \noindent\textit{(a)} According to \eqref{F0} and \eqref{F0''}, by using the growth of $f$
    and the change of variables $s=\frac{|x-y|}{t}$,
    \begin{align*}
        |&\Delta_\lambda^{-\sigma}f(x)|
        \leq C \int_0^\infty  \int_0^\infty P_t^\lambda(x,y) |f(y)| \,\frac{dy\, dt}{t^{1-2\sigma}} \\
           & \leq C\int_0^\infty\int_0^{2x} \frac{ty^\alpha}{(|x-y|+t)^2}  \frac{dy \,dt}{t^{1-2\sigma}}
          + C\int_0^\infty \int_{2x}^\infty \frac{t(xy)^\lambda y^\alpha}{(|x-y|+t)^{2(\lambda+1)}}
          \frac{dy\,dt}{t^{1-2\sigma}}  \\
        &=  C \int_0^\infty \frac{s^{-2\sigma}}{(s+1)^2}\,ds \int_0^{2x} \frac{y^\alpha}{|x-y|^{1-2\sigma}} \, dy
             + C\int_0^\infty \frac{s^{2\lambda-2\sigma}}{(s+1)^{2\lambda+2}}\,ds\int_{2x}^\infty \frac{x^\lambda
             y^{\alpha+\lambda}}{|x-y|^{2\lambda+1-2\sigma}} \, dy \\
        &\leq C x^{\alpha+2\sigma},
    \end{align*}
    because $0<2\sigma<1$ and $\alpha+2\sigma<\lambda$. Being $\alpha+2\sigma<\tlambda$, it is clear that
    $\Delta_\lambda^{-\sigma}f(x)$, as well as $f$, are in $L_{\widetilde{\lambda}/2}$.
    Then, applying Corollary~\ref{Lem permuta}, estimate \eqref{F1''} to justify Fubini's theorem and the semigroup
    property of $P_t^\lambda$ we get
    $$\partial_t P_t^\lambda(\Delta_\lambda^{-\sigma}f)(x)=\frac{1}{\Gamma(2\sigma)} \int_0^\infty \partial_t P_{t+s}^\lambda f(x) \,\frac{ds}{s^{1-2\sigma}}.$$
    This and Theorem~\ref{Th1.1}\textit{(iii)} for $f$ lead to
    $$\left| t \partial_t P_t^\lambda(\Delta_\lambda^{-\sigma}f)(x) \right|
        \leq C \|f\|_{C_+^\alpha}\int_0^\infty \frac{ts^{2\sigma-1}}{(t+s)^{1-\alpha}}\,ds
       =C \|f\|_{C^\alpha_+} t^{\alpha + 2\sigma},$$
    for all $x \in \R_+$. Hence $\Delta_\lambda^{-\sigma}f \in C^{\alpha+2\sigma}_+$ and
    $\|\Delta_\lambda^{-\sigma}f\|_{C^{\alpha+2\sigma}_+}\leq C \|f\|_{C^\alpha_+}$.

    \noindent\textit{(b)} Assume that $u \in C^\alpha_+\cap L_\sigma$,
    $0<\alpha-2\sigma<1$ and $\lambda \geq 1$.
    Let us first analyze the size of $\Delta_\lambda^\sigma u$.
    Recall \eqref{spectral positive}. We have
    \begin{align*}
        \int_0^\infty \left| P_t^\lambda u(x)-u(x) \right| \frac{dt}{t^{1+2\sigma}}
            = & \int_0^x \left| P_t^\lambda u(x)- u(x) \right| \frac{dt}{t^{1+2\sigma}}
               + \int_x^\infty \left| P_t^\lambda u(x)- u(x) \right| \frac{dt}{t^{1+2\sigma}} \\
            = & I_1(x) + I_2(x).
    \end{align*}
    We deal with $I_1$. Since $u\in C^\alpha_+$, Theorem~\ref{Th1.1} and  {Proposition}
    \ref{thm:harmonic extension} imply
    \begin{align}\label{x3}
        I_1(x)
            \leq & \int_0^x \int_0^t |\partial_s P_s^\lambda u(x)| \,\frac{ds\,dt}{t^{1+2\sigma}}
            \leq C \int_0^x \int_0^t s^{\alpha-1}\, \frac{ds\,dt}{t^{1+2\sigma}}
            \leq C x^{\alpha-2\sigma}.
    \end{align}
    In order to analyze $I_2$, we use \eqref{F0''} to deduce that
    \begin{align*}
        & \int_x^\infty |P_t^\lambda u (x)| \,\frac{dt}{t^{1+2\sigma}}
            \leq C \int_x^\infty \int_0^\infty \frac{t(xy)^\lambda}{(t+|x-y|)^{2(\lambda+1)}} | u (y)|
            \,\frac{dy\,dt}{t^{1+2\sigma}} \\
        & \qquad \leq C \int_x^\infty \left( \int_0^{x/2} \frac{y^\alpha(xy)^\lambda}{(t+x)^{2\lambda+1}} \,dy
                                           + \int_{x/2}^{2x} \frac{x^{2\lambda+\alpha}}{t^{2\lambda+1}} \,dy
                                           + \int_{2x}^\infty \frac{y^\alpha(xy)^\lambda}{(t+y)^{2\lambda+1}}\, dy  \right) \frac{dt}{t^{1+2\sigma}} \leq C x^{\alpha-2\sigma}.
    \end{align*}
    Then, since also $|u(x)|\leq Cx^\alpha$, we get
    \begin{equation}\label{x4}
        I_2(x)
            \leq C x^{\alpha-2\sigma}, \quad x \in \R_+.
    \end{equation}
    By plugging \eqref{x3} and \eqref{x4} into \eqref{spectral positive} it follows that $|\Delta_\lambda^\sigma u(x)|\leq Cx^{\alpha-2\sigma}$,
    for all $x\in\R_+$.

    Our next objective is to see that
    \begin{equation}\label{ob2}
        \left|s \partial_s P_s^\lambda\left(\Delta_\lambda^\sigma u\right)(x)\right|
            \leq Cs^{\alpha-2\sigma}, \quad x \in \R_+,~s>0.
    \end{equation}
    Once \eqref{ob2} is proved, we can apply Theorem~\ref{Th1.1} to conclude.
    By \eqref{F1''} and the size of $\Delta_\lambda^\sigma u(x)$,
    $$\int_0^\infty |\partial_s P_s^\lambda(z,x)|\int_0^\infty \left| P_t^\lambda u (x)- u (x) \right| \frac{dt}{t^{1+2\sigma}}\,dx
        \leq C \int_0^\infty \frac{x^{\alpha-2\sigma}}{s^2+(z-x)^2}\,dx < \infty,$$
for all $z\in\R_+$.  Also, we have that
    $$\int_0^\infty |\partial_s P_s^\lambda(z,x)|\int_0^t \left|\partial_rP_r^\lambda u(x)\right|\, dr\,dx<\infty,$$
    and
    $$\int_0^\infty |\partial_s P_s^\lambda(z,x)| \int_0^\infty \left|\partial_r
 P_r^\lambda(x,y)\right| |u(y)|\,dy\,dx<\infty.$$
    These estimates allow us to interchange the order of integration and write
    \begin{align*}
        \partial_s P_s^\lambda(\Delta_\lambda^\sigma u )(z)
            &=  C \int_0^\infty \partial_s P_s^\lambda[P_t^\lambda u-u](z)\, \frac{dt}{t^{1+2\sigma}} \\
            &=  C \int_0^s \int_0^t \partial_r^2 P_r^\lambda u(z)\big|_{r=s+\eta}\, d\eta\,\frac{dt}{t^{1+2\sigma}} \\
              &\quad  + C \int_s^\infty \left( \partial_r P_r^\lambda u(z)\big|_{r=t+s} -
              \partial_s P_s^\lambda u(z) \right) \frac{dt}{t^{1+2\sigma}}.
    \end{align*}
    By the estimates above and Theorem \ref{Th1.1},
    \begin{align*}
        \left|\int_0^s \int_0^t \partial_r^2 P_r^\lambda u(z)\big|_{r=s+\eta}\, d\eta\,\frac{dt}{t^{1+2\sigma}} \right|
            &\leq C \int_0^s \int_0^t (s+ \eta)^{\alpha-2}\, \frac{d \eta\,dt}{t^{1+2\sigma}} \\
       &    = C \int_0^s s^{\alpha-1} \int_0^{t/s} (1+r)^{\alpha-2} \,\frac{dr\,dt}{t^{1+2\sigma}} \\
            &\leq C s^{\alpha-1} \int_0^s  \frac{t}{s}\,\frac{dt}{t^{1+2\sigma}}
            \leq C s^{\alpha-2\sigma-1},
    \end{align*}
    and
    $$\left| \int_s^\infty \left( \partial_r P_r^\lambda u(z)\big|_{r=t+s} -
              \partial_s P_s^\lambda u(z) \right) \frac{dt}{t^{1+2\sigma}} \right|
        \leq C \int_s^\infty \big[ (t+s)^{\alpha-1} +s^{\alpha-1}  \big]\, \frac{dt}{t^{1+2\sigma}}\leq
 Cs^{\alpha-2\sigma-1}.$$
     Therefore \eqref{ob2} holds and this concludes the proof.
     \end{proof}

    \section{Proofs of technical results} \label{sec:demostraciones}

    Here we collect the proofs of some of the technical results left open in the previous sections.

\begin{proof}[Proof of estimate \eqref{csigma}]
    Observe that for $x\in\R_+$ we can write
    \begin{align*}
        C_\sigma^\lambda(x) &= \int_{|x-y|<1,\,y\in\R_+}\Bigg(K_\sigma^\lambda(x,y)-\frac{c_{\sigma}}{|x-y|^{1+2\sigma}}
        \Bigg)(x-y)\,dy \\
        &\quad+c_{\sigma}\lim_{\varepsilon\to0^+}\int_{\varepsilon<|x-y|<1,\,y\in\R_+}
        \frac{(x-y)}{|x-y|^{1+2\sigma}}\,dy =: I+II,
    \end{align*}
    where $c_\sigma$ is the constant that appears in the kernel for the fractional Laplacian on $\R$,
    namely,
    $c_\sigma=\frac{\sigma\Gamma(1/2+\sigma)}{\pi^{1/2}\Gamma(1-\sigma)}$.
    Let us begin with the first term above. By the definition of $K_\sigma^\lambda(x,y)$ we have
    $$I\leq C\int_0^\infty \chi_{\{|x-y|<1\}}(x,y)|x-y|\int_0^\infty\big|W_t^\lambda(x,y)-\W_t(x-y)\big|\frac{dt}{t^{1+\sigma}}\,dy.$$
    Except for the factor $\chi_{\{|x-y|<1\}}(x,y)|x-y|$, the double
    integral above is exactly like the term $A_1(x)$ in the proof of Lemma
    \ref{Lem:funcion}. In that proof such a term was splitted in two
    integrals $A_{1,1}(x)$ and $A_{1,2}(x)$. Let us call
    $\tilde{A}_{1,1}(x)$ and $\tilde{A}_{1,2}(x)$ the corresponding
    integrals with the extra factor
    $\chi_{\{|x-y|<1\}}(x,y)|x-y|$ in the integrand. We estimate both
    terms using that $|x-y|t^{-1/2}e^{-|x-y|^2/(4t)}\leq
    Ce^{-c|x-y|^2/t}$ in the following way. For the first one,
    \begin{align}
        \tilde{A}_{1,1}(x) &\leq C\int_0^\infty\left(\int_0^{x/2}~+\int_{(0,t/x)\cap(x/2,2x)}~+\int_{2x}^\infty~\right)
        \chi_{\{|x-y|<1\}}(x,y)e^{-c|x-y|^2/t}\,dy\,\frac{dt}{t^{1+\sigma}}\nonumber\\
        &\leq Cx\int_0^\infty e^{-cx^2/t}\frac{dt}{t^{1+\sigma}}+
        C\int_{x/2}^{2x}\int_{xy}^\infty~\frac{dt}{t^{1+\sigma}}\,dy
        +C\chi_{(0,1)}(x)\int_{2x}^{x+1}\int_0^\infty e^{-c|x-y|^2/t}\frac{dt}{t^{1+\sigma}}\,dy\nonumber\\
        &= \frac{Cx}{x^{2\sigma}}+\frac{C}{x^\sigma}\int_{x/2}^{2x}y^{-\sigma}dy+
        C\chi_{(0,1)}(x)\int_{2x}^{x+1}\frac{1}{|x-y|^{2\sigma}}\,dy\nonumber \\
        &\leq C\begin{cases}
        x^{1-2\sigma},&\hbox{when}~1/2<\sigma<1,\\
        1+\chi_{(0,1)}(x)|\ln x|,&\hbox{when}~\sigma=1/2;
        \end{cases}\label{teclado}
    \end{align}
    while for the second one,
    \begin{align*}
        \tilde{A}_{1,2}(x) &\leq C \int_0^\infty \int_{(t/x,\infty)\cap (0,x/2)}\frac{e^{-cx^2/t}}{y^2} \frac{dy\,dt}{t^{\sigma}}
        + C\chi_{(0,1)}(x)\int_0^\infty \int_{2x}^{x+1} e^{-c|x-y|^2/t} \frac{dy\,dt}{t^{\sigma+1}} \\
        &  \quad+ C \int_{x^2}^\infty \int_{(t/x,\infty)\cap (x/2,2x)}\frac{dy\,dt}{y^2t^{\sigma}}+
        C \int_0^{x^2} \int_{(t/x,\infty)\cap (x/2,2x)}\frac{dy\,dt}{x^2t^{\sigma}} \\
        &  \leq C x \int_0^\infty \frac{e^{-cx^2/t}}{t^{\sigma +1}} dt
        + C\chi_{(0,1)}(x)\int_{2x}^{x+1}\int_0^\infty \frac{e^{-c|x-y|^2/t}}{t^{\sigma + 1}} dt\,dy\\
        &\quad + Cx \int_{x^2}^\infty \frac{dt}{t^{\sigma + 1}}
        + \frac{C}{x} \int_0^{x^2} \frac{dt}{t^{\sigma}} \\
        &\leq\frac{Cx}{x^{2\sigma}}+C\chi_{(0,1)}(x)\int_{2x}^{x+1}\frac{1}{|x-y|^{2\sigma}}\,dy+\frac{Cx}{x^{2\sigma}}
        +\frac{Cx^{-2\sigma+2}}{x},
    \end{align*}
    which is bounded by \eqref{teclado}.
    This concludes the estimates for $I$. Notice that $II$ is zero when $x\geq1$.
    Using again this cancellation, when $x<1$ we get
$$ II= c_\sigma\int_{2x}^{x+1}\frac{(x-y)}{|x-y|^{1+2\sigma}}\,dy=c_\sigma\int_{2x}^{x+1}(y-x)^{-2\sigma}\,dy, $$
that is bounded by \eqref{teclado}.  Therefore \eqref{csigma} is true.
\end{proof}

 \begin{proof}[Proof of Lemma \ref{Lem intPoiss}]
        Let $m \in \N$ such that $m-1 \leq \beta < m$. By using subordination formula \eqref{subord kernel}
         we write, for $x,y\in\R_+$ and $t>0$,
        \begin{align*}
            \partial_t^m P_t^\lambda(x,y)
                &=  \partial_t^{m-1} \left( \frac{t}{\sqrt{2\pi}} \int_0^\infty \frac{e^{-u}}{u^{3/2}}\,
                \partial_v\left(W^\lambda_v(x,y) \right)\big|_{v=t^2/(4u)} \,du\right) \\
                &=   \partial_t^{m-1} \left( \frac{1}{\sqrt{\pi}} \int_0^\infty \frac{e^{-t^2/(4v)}}{\sqrt{v}} \,\partial_v W^\lambda_v(x,y) \,dv\right) \\
                &=  \frac{1}{\sqrt{\pi}} \int_0^\infty \partial_t^{m-1} \big(   e^{-t^2/(4v)}\big) \frac{1}{\sqrt{v}}\, \partial_v W^\lambda_v(x,y)\, dv.
        \end{align*}
        It is known that the Hermite polynomial of degree $k\geq0$ is given by
        $H_k(r)=(-1)^ke^{r^2}\partial_r^k\big(e^{-r^2}\big)$, for $r\in\R$. So
        $$\partial_t^m P_t^\lambda(x,y)
            = \frac{1}{\sqrt{\pi}} \Big(-\frac{1}{2} \Big)^{m-1} \int_0^\infty H_{m-1}\left(\frac{t}{2\sqrt{v}}\right)e^{-t^2/(4v)}
        \frac{1}{v^{m/2}}\,\partial_vW_v^\lambda(x,y)\,dv.$$
Thus, by using \eqref{F1'} and \eqref{F1''} to interchange the order of the integrals,
\begin{align*}
               \Big|\int_0^\infty &\partial_t^\beta P_t^\lambda(x,y)\, dy\Big|
                =  C\Big| \int_0^\infty \int_0^\infty
                        \partial_t^m P_{t+s}^\lambda(x,y) s^{m-\beta-1}\, ds \,dy \Big| \\
               & \leq  C
            \int_0^\infty \int_0^\infty\left| H_{m-1}\left(\frac{t+s}{2\sqrt{v}}\right)\right|e^{-(t+s)^2/(4v)} \frac{1}{v^{m/2}}\,
            \left|\int_0^\infty\partial_v W^\lambda_v(x,y)\,dy \right| s^{m-\beta-1} \,dv \,ds\\
         & \leq C \int_0^\infty \left(\int_0^\infty  \frac{s^{m-\beta-1}}{v^{m/2}} \,e^{-s^2/
         (8v)}\, ds \right) e^{-t^2/(8v)}
                        \left| \int_0^\infty \partial_v W^\lambda_v(x,y) \,dy \right| dv \\
               & \leq  C \int_0^\infty \frac{e^{-t^2/(8v)}}{v^{\beta/2}}\left| \int_0^\infty \partial_v W^\lambda_v(x,y) \,dy \right| dv,\quad x \in \R_+~\hbox{and}~t>0.
               \end{align*}
        By proceeding in a complete analogous way as in the proof of \cite[$(i) \Rightarrow (ii)$, Theorem 1.1, pp. 470--474]{BCFR1} we  can deduce that
        the integral involving the derivative of the heat kernel
        is bounded by
        \begin{align*}
           &  \left| \left( \int_0^{v/x}~+\int_{\min\{v/x,x/2\}}^{x/2}~ + \int_{\max\{v/x,x/2\}}^{\max\{v/x,3x/2\}}~
                        + \int_{\max\{v/x,3x/2\}}^\infty~ \right) \partial_v W^\lambda_v(x,y)\, dy \right| \nonumber\\
           & \quad  \leq  \int_0^{v/x} \left| \partial_v W^\lambda_v(x,y) \right|dy
                     +  \int_{\min\{v/x,x/2\}}^{x/2} \left| \partial_v W^\lambda_v(x,y) \right|dy
                     +  \int_{\max\{v/x,3x/2\}}^\infty \left| \partial_v W^\lambda_v(x,y) \right|dy \nonumber \\
           & \qquad +  \int_{(x/2,3x/2)\cap(v/x,\infty)} \left| \partial_v\left( W^\lambda_v(x,y) - \W_v(x-y) \right) \right|\,dy\nonumber\\
           &\qquad       +  \left| \partial_v \left( \int_{(x/2,3x/2)\cap(v/x,\infty)} \W_v(x-y)\, dy \right) \right| \nonumber\\
           & \qquad \leq C \left( \frac{e^{-x^2/8v}}{v} + \frac{\chi_{(0,3x^2/2)}(v)}{x^2}\right),
           \quad x,v \in \R_+.
        \end{align*}
        Hence, we get, for $x \in \R_+$ and $t>0$,
        \begin{align*}
            \left| \int_0^\infty \partial_t^\beta P_t^\lambda(x,y) \,dy \right|
                &\leq  C \left( \int_0^\infty \frac{e^{-(x^2+t^2)/(8v)}}{v^{(\beta+2)/2}} \,dv
                            + \frac{1}{x^2} \int_0^{3x^2/2} \frac{e^{-t^2/(8v)}}{v^{\beta/2}} \,dv    \right) \\
                &\leq  C \left( \frac{1}{(x+t)^\beta}
                            + \frac{1}{x^2} \int_0^{3x^2/2} \frac{e^{-t^2/(8v)}}{v^{(\beta+2-\delta)/2}} \,dv    \right)
                \leq  C t^{-\beta} \left(\frac{t}{x}\right)^\delta.
        \end{align*}
    \end{proof}

    \begin{proof}[Proof of Lemma \ref{Lem hankel}]
        Suppose that $ f  \in C_c(\R_+)$, the space of continuous functions with compact support on $\R_+$.
        For any $x\in\R_+$ and $t>0$ we have that $ P_t^\lambda f (x)= h_\lambda( e^{-yt} h_\lambda f (y))(x)$.
        Let $b>0$ such that $\supp f \subset (0,b)$. Since the function $\sqrt{z}J_\nu(z)$ is bounded on $\R_+$,
        for every $\nu>-1/2$, we deduce that for any $k\in\mathbb{N}$,
        \begin{align*}
            \int_0^\infty \left| \partial_t^k(e^{-yt}) \right| \left| h_\lambda f (y) \right|  \left| \sqrt{xy} J_{\lambda-1/2}(xy) \right| dy
                \leq C \frac{ b k!}{t^{k+1}} \| f \|_\infty.
        \end{align*}
        Then,  $\partial_t^k P_t^\lambda f (x)
            = h_\lambda ( (-1)^k y^k e^{-yt} h_\lambda f (y) )(x)$.
        Moreover, if $m \in \N$ and $m-1 \leq \beta < m$, we get
         \begin{align*}
            \int_0^\infty \int_0^\infty  \left| \sqrt{xy} J_{\lambda-1/2}(xy) \right|  y^m e^{-y(t+s)} \left| h_\lambda f (y)\right|   dy\, s^{m-\beta-1}\,ds
                \leq & C \frac{ b}{t^{\beta+1}} \|f\|_\infty,
        \end{align*}
        and then we can also write $\partial_t^\beta P_t^\lambda f(x)
                = h_\lambda (e^{-i \pi \beta} y^\beta e^{-yt}  h_\lambda f(y))(x)$.

        Let us now consider the general case $f \in L^2(\R_+)$. By using Corollary~\ref{Lem permuta} and \eqref{F1''} we obtain
        \begin{align*}
            \left| \partial_t^\beta P_t^\lambda f (x) \right|
                \leq & \left( \int_0^\infty \frac{1}{(t+|x-y|)^{2(\beta+1)}} \,dy\right)^{1/2} \|f\|_{L^2(\R_+)}
                \leq \frac{C}{t^{\beta+1/2}} \|f\|_{L^2(\R_+)}.
        \end{align*}
        Moreover, since the function $\sqrt{z}J_\nu(z)$ is bounded on $\R_+$,
        for every $\nu>-1/2$, Plancherel equality for the Hankel transformation (see \cite[(3)]{Ze}) leads to
        \begin{align*}
            & \left| h_\lambda\big( (-1)^\beta y^\beta e^{-ty} h_\lambda f(y) \big)(x) \right|
                \leq C \int_0^\infty y^\beta e^{-ty} |h_\lambda f (y)|\, dy \\
            & \qquad \leq C \left( \int_0^\infty y^{2\beta} e^{-2ty}\, dy \right)^{1/2} \|f\|_{L^2(\R_+)}
                     \leq \frac{C}{t^{\beta+1/2}} \|f\|_{L^2(\R_+)}.
        \end{align*}
        Hence, the operators $f\longmapsto\partial_t^\beta P_t^\lambda f$ and $f\longmapsto h_\lambda\big( (-1)^\beta y^\beta e^{-ty} h_\lambda f(y)\big)$, which are defined for every $t>0$,
                are bounded from $L^2(\R_+)$ into $L^\infty(\R_+)$. Since both operators coincide
                over the class  $C_c(\R_+)$ for every $t>0$, we conclude the result for a general $f \in L^2(\R_+)$
                by density.
   \end{proof}

\begin{proof}[Proof of Proposition \ref{Prop3}]
        According to Propositions~\ref{Prop1} and \ref{Prop2} we have that the following
        integral $\mathcal{I}$ is absolutely convergent and therefore we can write
        \begin{align}
            \mathcal{I}&= \int_0^\infty \int_0^\infty
                t^\beta \partial_t^\beta P_t^\lambda f(x) \overline{t^\beta \partial_t^\beta P_t^\lambda a(x)}\, \frac{dx\,dt}{t}
            = \lim_{N \to \infty} \int_{1/N}^N \int_0^\infty
                t^\beta \partial_t^\beta P_t^\lambda f(x) \overline{t^\beta \partial_t^\beta P_t^\lambda a(x)}\, \frac{dx\,dt}{t}\label{P1}.
        \end{align}
        The inner integral above can be written as
        \begin{align}\label{F3}
            \int_0^\infty t^\beta \partial_t^\beta P_t^\lambda  f(x)\overline{ t^\beta \partial_t^\beta P_t^\lambda a(x)}\, dx
                = \int_0^\infty f(y) \int_0^\infty t^\beta \partial_t^\beta P_t^\lambda(x,y)\overline{ t^\beta \partial_t^\beta P_t^\lambda a(x)}\, dx\, dy,
        \end{align}
        We need to justify the interchange of the order of integration in \eqref{F3}. We denote by $I$ a bounded interval in
        $\R_+$ such that $\supp a \subset I$. Notice that if $0<x<x_I+|I|$, where $x_I$ denotes the center of $I$, then
        $$\int_I \frac{1}{(t+|x-z|)^{\beta+1}}\,dz
            \leq \frac{C}{t^{\beta+1}}
            \leq \frac{C}{t^{\beta+1}(1+x)^{\beta+1}}.$$
        Moreover, if $x \geq x_I + |I|$ and $z \in I$, $|x-z| \sim |x-x_I|$. Hence, if  $x \geq x_I + |I|$ ,
        $$\int_I \frac{1}{(t+|x-z|)^{\beta+1}}\,dz
            \leq \frac{C}{|x-x_I|^{\beta+1}}
            \leq \frac{C}{(1+x)^{\beta+1}}.$$
Therefore, by \eqref{F1''} we obtain,
\begin{equation}\label{F4}
       \left| t^\beta \partial_t^\beta P_t^\lambda a(x) \right|
            \leq  \int_I \frac{t^\beta\|a\|_\infty}{(t+|x-z|)^{\beta+1}}\,dz\leq \frac{\zeta(t)}{(1+x)^{\beta+1}},
            \end{equation}
            where $\zeta(t)=C\max\{t^{-\beta-1},1\}$, and $C$ is a constant that depends only on $I$.
        By using again \eqref{F1''}, \eqref{F4} leads to
        \begin{align*}
            &\int_0^\infty |f(y)| \int_0^\infty
                    \left| t^\beta \partial_t^\beta P_t^\lambda(x,y) \right| \left| t^\beta \partial_t^\beta P_t^\lambda a(x) \right| \,dx\, dy \\
            & \leq \tilde{\zeta}(t) \int_0^\infty |f(y)| \left(\int_0^{y/2}~ + \int_{y/2}^\infty~\right)
            \frac{t^\beta }{\big[(t+|x-y|)(1+x)\big]^{\beta+1}}\,  dx \,dy \\
            & \leq \tilde{\zeta}(t) \Bigg( \int_0^\infty \frac{|f(y)|}{\left[\min\{t,1\}(1+y)\right]^{\beta+1}}
            \,dy \int_0^\infty \frac{dx}{(1+x)^{\beta+1}}
                                           + \int_0^\infty \frac{|f(y)|}{(1+y)^{\beta+1}}\, dy \int_\R \frac{t^\beta}{(t+z)^{\beta+1}}  \,dz    \Bigg)\\
            &  \leq \tilde{\zeta}(t) \int_0^\infty \frac{|f(y)|}{(1+y)^{\beta+1}} \,dy.
        \end{align*}
        Here $\tilde{\zeta}(t)$ is a positive and continuous function of $t>0$. This proves that
        the interchange in the order of integration in \eqref{F3} is
        justified and that the integral in the right-hand side is absolutely convergent
        and also we can write, for every $N \in \N$,
        \begin{multline}\label{P3}
            \int_{1/N}^{N} \int_0^\infty t^\beta \partial_t^\beta P_t^\lambda f(x) \overline{t^\beta \partial_t^\beta P_t^\lambda a(x)}\, \frac{dx\, dt}{t}\\
                = \int_0^\infty f(y) \int_{1/N}^{N} \int_0^\infty t^\beta \partial_t^\beta P_t^\lambda(x,y)
                \overline{t^\beta \partial_t^\beta P_t^\lambda a(x)}\, \frac{dx \,dt}{t}\,dy=:
                 \int_0^\infty f(y)\, \mathcal{I}_N(y)\,dy.
        \end{multline}
        It is clear that by Lemma~\ref{Lem intPoiss} we have that
        \begin{equation*}\label{A2}
            \left| t^\beta \partial_t^\beta P_t^\lambda  a (x) \right|
                \leq C \left(\frac{t}{x}\right)^\delta,
        \end{equation*}
        for some $0<\delta<\beta \wedge 1/2$.  Then, Corollary \ref{Lem permuta} and Lemma \ref{Lem hankel} imply that
        \begin{align}\label{P5}
           \int_0^\infty &t^\beta \partial_t^\beta P_t^\lambda(x,y) \overline{t^\beta \partial_t^\beta P_t^\lambda a(x)}\,dx
                = \left[ s^\beta \partial_s^\beta P_s^\lambda\left(\overline{ t^\beta \partial_t^\beta P_t^\lambda a} \right)(y)\right]\Bigg|_{s=t} \nonumber \\
           &    =   \left[ s^\beta t^\beta h_\lambda \left(e^{-2\pi i \beta} z^{2\beta} e^{-(s+t)z} \overline{h_\lambda a} \right)(y)\right]\Bigg|_{s=t}
                =  \overline{t^{2\beta}\partial_t^{2\beta} P_{2t}^\lambda a(y)}.
        \end{align}
        Again, by applying Lemma~\ref{Lem hankel} and interchanging  the order of integration we get,
        for every $t,y \in \R_+$ and $N \in \N$,
        \begin{align*}
            \mathcal{I}_N(y)= h_\lambda \left[e^{-2\pi i \beta} \left(\int^{N}_{1/N} (tz)^{2\beta-1}e^{-2tz}z\,dt \right) h_\lambda a(z) \right](y).
        \end{align*}
        Since
        $$ \left(\int^{N}_{1/N} (tz)^{2\beta-1}e^{-2tz}z\,dt \right) h_\lambda a(z)
            \longrightarrow \frac{\G(2\beta)}{2\beta} h_\lambda(a)(z), \quad
           \hbox{as}~N \to \infty,~\hbox{in}~L^2(\R_+),$$
        we get
        $\mathcal{I}_N(y)  \longrightarrow \frac{e^{2\pi i \beta}\G(2\beta)}{2\beta}a(y)$, as $N \to \infty$,  in $L^2(\R_+)$. Then, there exists an increasing sequence $(N_k)_{k \in \N} \subset \N$ such that         \begin{equation}\label{P4}
            \mathcal{I}_{N_k}(y) \longrightarrow \frac{e^{2\pi i \beta}\G(2\beta)}{2\beta}a(y), \quad \text{as } k \to \infty, \text{ a.e. } y \in \R_+.
        \end{equation}

        \medskip

\noindent \textsc{Claim.} \textit{For all $y\in\R_+$ we have}
          $$ \sup_{A>0} \left| \int_A^\infty t^{2\beta} \partial_t^{2\beta} P_{2t}^\lambda a(y) \,\frac{dt}{t} \right|
                \leq \frac{C}{(1+y)^{\lambda+1}}.$$

\begin{proof}[Proof of the Claim.]
        For the proof we first note that
        \begin{align*}
            \int_A^\infty t^{2\beta} \partial_t^{2\beta} P_{2t}^\lambda a(y)\, \frac{dt}{t}
                = & \int_A^\infty t^{2\beta-1} \int_I \partial_t^{2\beta} P_{2t}^\lambda(y,z) a(z) \,dz \,dt \\
                = & \int_I  a(z) \int_A^\infty t^{2\beta-1} \partial_t^{2\beta} P_{2t}^\lambda(y,z) \,dt\,dz.
        \end{align*}
        The interchange in the order of integration is justified by \eqref{F1}.
        To get the estimate in the claim we need to show that
        \begin{equation}\label{H1}
            \sup_{A>0} \left| \int_I a(z)  \int_A^\infty t^{2\beta-1} \partial_t^{2\beta} P_{2t}^\lambda(y,z) \,dt\, dz\right|
                \leq \frac{C}{(1+y)^{\lambda+1}}.
        \end{equation}
        In order to do this we have to distinguish four cases. This technique was also used
        in the proof of Lemma 5.1 in \cite{MSTZ1}.

        \smallskip

        \noindent \textbf{Case 1.} $2\beta<1$.
        Straightforward manipulations lead to
        \begin{align*}
            \int_A^\infty t^{2\beta-1} \partial_t^{2\beta} P_{2t}^\lambda(y,z) \,dt
                &=  2^{\lfloor 2\beta \rfloor-2\beta+1} \int_{2A}^\infty t^{2\beta-1} \partial_t^{2\beta} P_{t}^\lambda(y,z) \,dt \\
                &=  \frac{2^{\lfloor 2\beta \rfloor-2\beta+1}e^{-i\pi(1-\beta)}}{\G(1-\beta)}
                    \int_{2A}^\infty t^{2\beta-1} \int_0^\infty \partial_t P_{t+s}^\lambda(y,z) s^{-2\beta}\, ds\, dt \\
                &=  \frac{2^{\lfloor 2\beta \rfloor-2\beta+1}e^{-i\pi(1-\beta)}}{\G(1-\beta)}
                    \int_{2A}^\infty t^{2\beta-1} \int_t^\infty \partial_u P_{u}^\lambda(y,z) (u-t)^{-2\beta}\, du\, dt.
        \end{align*}
        Here $\lfloor \gamma \rfloor$ represents the biggest integer less than $\gamma$.
        By \eqref{F1} we get
        \begin{align*}
           & \int_{2A}^\infty t^{2\beta-1} \int_t^\infty \left|\partial_u P_{u}^\lambda(y,z)\right| (u-t)^{-2\beta} \,du\, dt
                \leq  C \int_{2A}^\infty t^{2\beta-1} \int_t^\infty u^{-\tlambda -1} (u-t)^{-2\beta}\, du \,dt \\
           & \qquad \qquad \leq  C \int_{2A}^\infty t^{-\tlambda-1} dt \int_1^\infty v^{-\tlambda -1} (v-1)^{-2\beta} \,dv
           < \infty.
        \end{align*}
        Then, integrating by parts,
        \begin{align*}
           \int_A^\infty t^{2\beta-1} \partial_t^{2\beta} P_{2t}^\lambda(y,z) \,dt
           &     =  \frac{2^{\lfloor 2\beta \rfloor-2\beta+1}e^{-i\pi(1-\beta)}}{\G(1-\beta)}
                    \int_{2A}^\infty t^{2\beta-1} \int_t^\infty \partial_u P_{u}^\lambda(y,z) (u-t)^{-2\beta} \,du \,dt \\
           &    = \frac{2^{\lfloor 2\beta \rfloor-2\beta+1}e^{-i\pi(1-\beta)}}{\G(1-\beta)}
                    \int_{2A}^\infty \partial_u P_{u}^\lambda(y,z) \int_{2A/u}^1 \frac{s^{2\beta-1}}{(1-s)^{2\beta}}\, ds \,du \\
           &      = \frac{2^{\lfloor 2\beta \rfloor-2\beta+1}e^{i\pi\beta}}{\G(1-\beta)}
                    \int_{2A}^\infty  P_{u}^\lambda(y,z) \left( \frac{ 2A }{u-2A} \right)^{2\beta}\,\frac{ du}{u} \\
           &      = \frac{2^{\lfloor 2\beta \rfloor-2\beta+1}e^{i\pi\beta}}{\G(1-\beta)}
                    \left( \int_{2A}^{3A+|y-z|} + \int_{3A+|y-z|}^\infty \right) P_{u}^\lambda(y,z)
                    \left( \frac{ 2A }{u-2A} \right)^{2\beta}\,\frac{ du}{u} \\
           &         = J_1(A,y,z) +  J_2(A,y,z).
        \end{align*}

        According to \eqref{F0} we deduce that, for each $A,y,z \in \R_+$,
        \begin{align*}
            \left| J_1(A,y,z) \right|
                \leq & C \int_{2A}^{3A+|y-z|} \left( \frac{ 2A }{u-2A} \right)^{2\beta}
                    \frac{u}{(u+|y-z|)^2} \left( \frac{y \wedge z}{u+|y-z|} \wedge 1 \right)^\lambda \frac{ du}{u} \\
                \leq & C \frac{A^{2\beta}}{(A+|y-z|)^2}  \left( \frac{y \wedge z}{A+|y-z|} \wedge 1 \right)^\lambda
                    \int_{2A}^{3A+|y-z|}  \frac{ du }{(u-2A)^{2\beta}} \\
                \leq & C \frac{A^{2\beta}}{(A+|y-z|)^{2\beta+1}}  \left( \frac{y \wedge z}{A+|y-z|} \wedge 1 \right)^\lambda,
        \end{align*}
        and also
        \begin{align*}
            \left| J_2(A,y,z) \right|
                \leq & C \int_{3A+|y-z|}^\infty \left( \frac{ 2A }{u-2A} \right)^{2\beta}
                    \frac{u}{(u+|y-z|)^2} \left( \frac{y \wedge z}{u+|y-z|} \wedge 1 \right)^\lambda \frac{ du}{u} \\
                \leq & C \frac{A^{2\beta}}{(A+|y-z|)^{2\beta}}  \left( \frac{y \wedge z}{A+|y-z|} \wedge 1 \right)^\lambda
                    \int_{3A+|y-z|}^\infty  \frac{ du }{(u+|z-y|)^{2}} \\
                \leq & C \frac{A^{2\beta}}{(A+|y-z|)^{2\beta+1}}  \left( \frac{y \wedge z}{A+|y-z|} \wedge 1 \right)^\lambda.
        \end{align*}
        Hence,
        \begin{equation}\label{B1}
            \left| \int_A^\infty t^{2\beta-1} \partial_t^{2\beta} P_{2t}^\lambda(y,z) \,dt \right|
                \leq C \frac{A^{2\beta}}{(A+|y-z|)^{2\beta+1}}  \left( \frac{y \wedge z}{A+|y-z|} \wedge 1 \right)^\lambda,
        \end{equation}
        which gives \eqref{H1}.

        \smallskip

\noindent \textbf{Case 2.} $\beta=1/2$. By \eqref{F0},
        \begin{equation}\label{B2}
            \left| \int_A^\infty t^{2\beta-1} \partial_t^{2\beta} P_{2t}^\lambda(y,z)\, dt \right|
                =  C P_{2A}^\lambda(y,z)
                \leq  C \frac{A}{(A+|y-z|)^{2}}  \left( \frac{y \wedge z}{A+|y-z|} \wedge 1 \right)^\lambda.
        \end{equation}

        \smallskip

\noindent \textbf{Case 3.} $\beta=k/2$, for $k \in \N$, $k \geq 2$.  Integration by parts leads to
        \begin{align*}
            & \int_A^\infty t^{2\beta-1} \partial_t^{2\beta} P_{2t}^\lambda(y,z) \,dt
                = 2 \int_{2A}^\infty t^{k-1} \partial_t^{k} P_{t}^\lambda(y,z)\, dt
                = \sum_{j=0}^{k-1} C_j A^j \partial_t^j P_t^\lambda(y,z) \big|_{t=2A},
        \end{align*}
        for certain $C_j \in \R$ , $j=0, \dots, k-1$. By using \eqref{F0'} we get
        $$\left| \int_A^\infty t^{2\beta-1} \partial_t^{2\beta} P_{2t}^\lambda(y,z)\, dt \right|
            \leq C \sum_{j=0}^{k-1} \frac{A^j}{(A+|y-z|)^{\beta+1}} \left( \frac{y \wedge z}{A+|y-z|} \wedge 1 \right)^\lambda. $$

            \smallskip

 \noindent \textbf{Case 4.  $k-1<2\beta<k$, for $k \in \N$, $k \geq 2$.} Integrating by parts again we obtain that
        \begin{align*}
           \int_A^\infty t^{2\beta-1} \partial_t^{2\beta} P_{2t}^\lambda(y,z) \,dt
             = &  \frac{2^{k-2\beta}e^{-i\pi(k-\beta)}}{\G(k-\beta)}
                    \int_{2A}^\infty  \partial_u^k P_{u}^\lambda(y,z) \int_{2A/u}^1 \frac{u^{k-1}s^{2\beta-1}}{(1-s)^{2\beta+1-k}}\, ds\, du  \\
&              =  \sum_{j=0}^{k-1} C_j \int_{2A}^\infty u^{j-k}
\partial_u^j P_{u}^\lambda(y,z)
\frac{A^{2\beta}}{(u-2A)^{2\beta+1-k}} \,du,
        \end{align*}
        where $C_j \in \R$ , $j=0, \dots, k-1$. Suppose that $j=1, \dots, k-1$. By using \eqref{F0'} we deduce that
        \begin{align*}
            & \left| \int_{2A}^\infty u^{j-k} \partial_u^j P_{u}^\lambda(y,z) \frac{A^{2\beta}}{(u-2A)^{2\beta+1-k}} \,du \right|
                 \leq \left| \int_{2A}^{3A} u^{j-k} \partial_u^j P_{u}^\lambda(y,z) \frac{A^{2\beta}}{(u-2A)^{2\beta+1-k}}\, du \right| \\
            & \qquad \qquad \qquad + \left| \int_{3A}^\infty u^{j-k} \partial_u^j P_{u}^\lambda(y,z) \frac{A^{2\beta}}{(u-2A)^{2\beta+1-k}}\, du \right| \\
            & \qquad \qquad \leq C \Bigg[ \frac{A^{2\beta+j-k}}{(A+|y-z|)^{j+1}}
            \left( \frac{y \wedge z}{A+|y-z|} \wedge 1 \right)^\lambda \int_{2A}^{3A} \frac{du}{(u-2A)^{2\beta+1-k}}  \\
            & \qquad \qquad \qquad + \frac{A^{2\beta}}{(A+|y-z|)^{j+1}} \left( \frac{y \wedge z}{A+|y-z|}
            \wedge 1 \right)^\lambda \int_{3A}^\infty \frac{du}{u^{k-j}(u-2A)^{2\beta+1-k}}  \Bigg] \\
            & \qquad \qquad \leq C  \frac{A^{j}}{(A+|y-z|)^{j+1}} \left( \frac{y \wedge z}{A+|y-z|} \wedge 1 \right)^\lambda.
        \end{align*}
        Also, by \eqref{F0} we get
        \begin{align*}
            & \left| \int_{2A}^\infty u^{-k} P_{u}^\lambda(y,z) \frac{A^{2\beta}}{(u-2A)^{2\beta+1-k}} \,du \right|\\
            &\qquad \qquad
                 \leq C \frac{A^{2\beta}}{(A+|y-z|)^{2}}
                 \left( \frac{y \wedge z}{A+|y-z|} \wedge 1 \right)^\lambda \int_{2A}^\infty \frac{u^{1-k}}{(u-2A)^{2\beta+1-k}} \,du \\
            & \qquad \qquad \leq C  \frac{A}{(A+|y-z|)^{2}} \left( \frac{y \wedge z}{A+|y-z|} \wedge 1 \right)^\lambda.
        \end{align*}
        Hence, we obtain that
        \begin{equation}\label{B3}
             \left| \int_A^\infty t^{2\beta-1} \partial_t^{2\beta} P_{2t}^\lambda(y,z)\,dt  \right|
                \leq C \frac{A}{(A+|y-z|)^{2}} \left( \frac{y \wedge z}{A+|y-z|} \wedge 1 \right)^\lambda.
        \end{equation}
        Let $\delta \in (0,1]$ and suppose that $I \subset (0,c]$. We can write
        \begin{align}\label{laststep}
            & \int_I |a(z)| \frac{A^\delta}{(A+|y-z|)^{\delta+1}} \left( \frac{y \wedge z}{A+|y-z|} \wedge 1 \right)^\lambda dz\\
      &     \qquad \qquad     \leq \|a\|_\infty  \int_I \frac{A^\delta}{(A+|y-z|)^{\delta+1}} \left( \frac{y \wedge z}{A+|y-z|} \wedge 1 \right)^\lambda dz \nonumber \\
            & \qquad \qquad \leq C \left( A^\delta \chi_{(0,2c)}(y) \int_0^c \frac{1}{(A+|y-z|)^{\delta+1}}\, dz
                                            + \chi_{(2c,\infty)}(y) \int_0^c \frac{z^\lambda}{(A+|y-z|)^{\lambda+1}} \,dz\right) \nonumber\\
            & \qquad \qquad \leq C \left( A^\delta \chi_{(0,2c)}(y) \int_0^\infty \frac{du}{(A+u)^{\delta+1}}
                                            + \frac{\chi_{(2c,\infty)}(y)}{y^{\lambda+1}} \right) \nonumber\\
            & \qquad \qquad \leq C \left( \chi_{(0,2c)}(y) + \frac{\chi_{(2c,\infty)}(y)}{y^{\lambda+1}} \right)
                \leq \frac{C}{(1+y)^{\lambda+1}}.\nonumber
        \end{align}
        From \eqref{B1}, \eqref{B2}, \eqref{B3} and \eqref{laststep} we deduce \eqref{H1}. Hence the
        claim is proved.
\end{proof}

        Finally, by taking into account \eqref{P1}, \eqref{F3}, \eqref{P3}, \eqref{P5},
        \eqref{P4} and applying the dominated convergence theorem, Proposition \ref{Prop3} follows.
    \end{proof}
    
\medskip

\noindent{\textbf{Acknowledgements.} We are grateful to the referees for comments that helped us
to improve the presentation of the paper.}



\end{document}